\documentclass[10pt,reqno]{amsart}
\usepackage{amsmath,amsfonts,amssymb,mathrsfs}
\usepackage{amssymb,mathrsfs,graphicx}
\usepackage{color}
\usepackage{fullpage}

\usepackage{ifthen} 
\provideboolean{shownotes} 
\setboolean{shownotes}{true} 
\newcommand{\margnote}[1]{
\ifthenelse{\boolean{shownotes}}%
{\marginpar{\raggedright\tiny\texttt{#1}}}%
{}%
}
\newcommand{\hole}[1]{
\ifthenelse{\boolean{shownotes}}%
{\begin{center} \fbox{ \rule {.25cm}{0cm}
\rule[-.1cm]{0cm}{.4cm} \parbox{.85\textwidth}{\begin{center}
\texttt{#1}\end{center}} \rule {.25cm}{0cm}}\end{center}}
{}
}

 \title[Relaxation to gradient flows]{From gas dynamics with large friction to gradient flows describing diffusion theories}

\author{Corrado Lattanzio}
\address[Corrado Lattanzio]{\newline 
Dipartimento di Ingegneria e Scienze dell'Informazione e Matematica
\newline
Universit\`a degli Studi dell'Aquila
\newline
Via Vetoio
\newline
I-67010 Coppito (L'Aquila) AQ 
\newline 
Italy
}
\email{corrado@univaq.it}

\author{Athanasios E. Tzavaras}
\address[Athanasios E. Tzavaras]{
\newline
Computer, Electrical and Mathematical Science and Engineering Division 
\newline
King Abdullah University of Science and Technology (KAUST)
\newline 
Thuwal 23955-6900,  Saudi Arabia
}
\email{athanasios.tzavaras@kaust.edu.sa}

\numberwithin{equation}{section}

\newtheorem{theorem}{Theorem}[section]
\newtheorem{lemma}[theorem]{Lemma}

\newtheorem{proposition}[theorem]{Proposition}
\newtheorem{definition}[theorem]{Definition}
\theoremstyle{remark}
\newtheorem{remark}[theorem]{Remark}

\newcommand{\X}{C_{\mathrm{x}}}
\newcommand{\C}{C_\kappa}
\newcommand{\R}{\mathbb{R}}
\newcommand{\I}{\mathbb{I}}
\newcommand{\T}{\mathbb T}

\newcommand{\e}{\varepsilon}
\newcommand{\eps}{\varepsilon}
\newcommand{\cF}{\mathcal{F}}
\newcommand{\dx}{\mathop{{\rm div}_{x}}}

\newcommand{\pxi}{\partial_{x_{i}}}
\newcommand{\cE}{\mathcal{E}}
\newcommand{\brho}{\bar\rho}
\newcommand{\baru}{\bar u}
\newcommand{\barm}{\bar m}

\def\charf {\mbox{{\text 1}\kern-.30em {\text l}}}

\def\del{\partial}

\begin{document}
\allowdisplaybreaks







\begin{abstract}
We study the emergence of gradient flows in Wasserstein distance as high friction limits of an abstract Euler flow
generated by an energy functional. We develop a relative energy calculation that connects  the Euler flow to the gradient flow
in the diffusive limit regime. We apply this approach to prove convergence from the Euler-Poisson system with friction
to the Keller-Segel system in the regime that the latter has smooth solutions. The same methodology is used to 
establish convergence from the Euler-Korteweg theory with monotone pressure laws to the Cahn-Hilliard equation.
\end{abstract}

\maketitle


\tableofcontents

\section{Introduction}

Following the works of  Jordan-Kinderlehrer-Otto \cite{JKO98} and  Otto \cite{Ott01} a large interest 
was generated for diffusive equations induced as gradient flows of functionals in  the form:
\begin{equation}\label{eq:intro1}
\rho_t - \dx \left(\rho\nabla_x \frac{\delta \mathcal{E}(\rho)}{\delta \rho }\right ) =0\, .
\end{equation}
A key novelty of the approach introduced in these papers is the  use of the Wasserstein space of probability measures as a framework where the gradient flow is considered; for a complete theory we refer to the monograph \cite{AGS05}.

The objective of this work is to explore the induction of such diffusion problems as high friction limits
of abstract Euler flows of the form
\begin{equation}\label{eq:intro2}
\begin{cases}
\displaystyle{\frac{\del \rho}{\del t} + \dx (\rho u) = 0 } & \\
& \\
\displaystyle{\rho \frac{\partial u}{\partial t} + \rho u \cdot \nabla_x u= - \rho \nabla_x \frac{\delta \cE}{\delta \rho} - \zeta \rho u}\, , &
\end{cases}
\end{equation}
where $\zeta > 0$ is a (large)  friction coefficient $\zeta>0$ and  $\cE(\rho)$ is a functional on the density that generates the evolution.
This problem is introduced in \cite{GLT15} (the Hamiltonian flow case $\zeta=0$) with the objective to put in a common framework several commonly used systems in applications,
like the Euler equations, the Euler-Poisson system and the Euler-Korteweg theory. In this work we study the emergence of the system \eqref{eq:intro1}
from the  system \eqref{eq:intro2} in the high friction regime $\zeta \to \infty$.

This type of problem belongs to the general realm of diffusive limits, which has been addressed in various contexts with several techniques; 
we refer to \cite{DM04} for a survey. The simplest example  of a high friction limit  that fits within the present functional framework 
(from \eqref{eq:intro2} to  \eqref{eq:intro1})   is the limit from the Euler system with friction  to the porous media equation. This has been addressed
again with various methodologies, see {\it e.g.}  \cite{MM90,HMP05,HPW11} and in particular \cite{LT12a} using  the relative energy method adopted here.

We develop a general methodology for treating the diffusive limit  from \eqref{eq:intro2} to  \eqref{eq:intro1} and 
apply it to two examples:  First, we consider  generalized Keller-Segel type models 
\begin{equation}
\label{intro-KS}
    \begin{cases}
	\displaystyle{\rho_{t} =  \dx \big ( \nabla_{x} p(\rho)  -  \X\rho \nabla_{x} c   \big ) }&   \\[4pt]
	- \triangle_{x} c + \beta c = \rho  - <\rho>.
    \end{cases}
\end{equation}   
as high friction imits of  the Euler-Poisson system  with attractive potentials ($\X>0$) and friction:
  \begin{equation}
\label{intro-EP}
    \begin{cases}
	\displaystyle{\rho_{t} + \dx m =0}&   \\[4pt]
       \displaystyle{ m_{t} +  \dx \frac{m\otimes m}{\rho} 
       +  \nabla_{x}p(\rho) 
	= - \zeta m} + {\X}\rho \nabla_{x} c & \\[4pt]
	-\triangle_{x} c + \beta c=  \rho - < \rho >.
    \end{cases}
\end{equation} 
This example corresponds to the choice of the functional 
\begin{equation*}
\mathcal{E}(\rho) = \int \left ( h(\rho) - \tfrac{1}{2}\X\rho c\right )dx\, ,
\end{equation*}
where $h$ and $p$ are linked by the thermodynamic consistency relations $\rho h''(\rho) = p'(\rho)$, $\rho h'(\rho) = p(\rho) + h(\rho)$, while
$c$  is the solution of the Poisson equation 
\begin{equation*}
- \triangle_{x} c + \beta c = \rho \, - <\rho>, \quad <\rho> = \int \rho dx,\ \beta\geq 0\, ,
\end{equation*}
normalized by requiring $<c> = 0$ for $\beta = 0$.
  For alternative methodologies on this problem see \cite{DiFD10, LT12b}; related models 
in the context of semiconductors devices  with repulsive potentials $\X < 0$  are analyzed in \cite{MN95a,LM99,Lat00}. 
For a study of the limiting Keller-Segel model \eqref{intro-KS} as a gradient flows we refer to \cite{BCC08}.
   
 As a second paradigm entering into this framework,
we consider the Euler-Korteweg system with friction 
   \begin{equation}
   \label{intro-EK}
    \begin{cases}
	\displaystyle{\rho_{t} +\dx m =0} &   \\[6pt]
       \displaystyle{ m_{t} + \dx \frac{m\otimes m}{\rho} 
	= -\zeta m -  \rho \nabla_x \big (h'(\rho) - C_\kappa \triangle_x\rho \big)}& \\
    \end{cases}
\end{equation}
converging in the high-friction regime to the Cahn-Hilliard equation
\begin{equation}
\label{intro-CH}
	\rho_{t} =\dx \left (\rho \nabla_{x}\big (h'(\rho) - C_\kappa \triangle_x\rho \big) \right )  = 
	\dx \big ( \nabla_{x}p(\rho) - C_\kappa\rho  \nabla_{x}\triangle_x\rho \big )\, ,
\end{equation}
which  corresponds to the choice of functional
\begin{equation*}
\mathcal{E}(\rho) = \int \left (h(\rho) +\tfrac12 C_\kappa |\nabla_x \rho |^2 \right ) dx, \quad C_\kappa>0\, .
\end{equation*}

The technical tool  consists of  a functional form of the relative energy identity  introduced in \cite{GLT15}, and inspired by \cite{LT12a,LT12b}
and the relative energy calculations of  Dafermos \cite{Dafermos79,Dafermos79a}. 
The relative energy  monitors the distance between  solutions in appropriate norms pertinent  to the aforementioned equations \eqref{eq:intro2} to  \eqref{eq:intro1}.
It provides a very efficient tool to carry out the limiting process, as it is precisely adapted to the functional
framework of both problems \eqref{eq:intro1} and  \eqref{eq:intro2}.

The outline of this work is as follows. 
In Section \ref{sec:largefrictvsgradflows} we introduce the  relative kinetic energy 
\begin{equation*}
K ( \rho, m | \brho, \barm ) : =   \frac{1}{2} \int \rho | u - \baru|^2 \, dx\, ,
\end{equation*} 
and the relative potential energy
\begin{equation*}
\cE (\rho | \brho) := \cE(\rho) - \cE (\brho) - \big \langle \frac{\delta \cE}{\delta \rho} (\brho) , \rho - \brho \big \rangle\, ,
\end{equation*} 
and use them to derive an identity for the distance between two solutions  of \eqref{eq:intro2}; see \eqref{reltote} in Section \ref{subset:twohypsol}.
The same tool is used in order to measure the distance between solutions of \eqref{eq:intro2} and \eqref{eq:intro1} in Section \ref{subset:1hyp1gf}. 
It provides a yardstick to measure the distance in  the relaxation limit. The identity carries seamlessly to the limit and 
yields an identity between two solutions of \eqref{eq:intro1} in terms of the relative potential energy; see \eqref{relpotegrad} in Section \ref{subset:limit only}.
After this formal calculation, we study the relaxation limits from weak solutions of the hyperbolic relaxing model  toward strong solutions of the diffusive equations.
This is carried out in two cases:  from the  Euler-Poisson systems with attractive potentials towards  Keller-Segel type models in Section \ref{sec:KS}, and from 
the Euler-Korteweg system with friction toward  the Cahn-Hilliard equation  in Section \ref{sec:korteweg}.

%
%
%
\section{A large friction theory converging toward gradient flows}\label{sec:largefrictvsgradflows}
We start our analysis by presenting a relaxation theory of large friction converging  towards gradient flow dynamics. 
This formalism  will unify  in a common framework the results on covergence from the compressible Euler system with friction to the porous media equation 
obtained in \cite{LT12a}, 
with convergence results towards Keller-Segel type systems  (see  \cite{LT12b} for preliminary results in this direction), 
or towards the Cahn-Hilliard equation, obtained in the following sections.
The specific cases will be obtained as particular examples of the general framework  via an appropriate choice of the entropy functional defining the flow of the limiting equation.

To this aim, let us consider the following system of equations consisting of a conservation of mass and a functional momentum equation
\begin{equation}
\label{funsys}
\begin{cases}
\displaystyle{\frac{\del \rho}{\del t} + \dx (\rho u) = 0 } & \\
& \\
\displaystyle{\rho \frac{\partial u}{\partial t} + \rho u \cdot \nabla_x u= - \rho \nabla_x \frac{\delta \cE}{\delta \rho} - \zeta \rho u}\, , &
\end{cases}
\end{equation}
where $\rho \ge  0$ is the density and $u$ is the velocity. 
Moreover,  $\displaystyle{\frac{\delta \cE}{\delta \rho}}$ stands for  the generator of the
first variation of the functional $\cE(\rho)$ (see the discussion in \cite{GLT15}), and   
the term $-\zeta \rho u$   accounts for a damping force with
frictional coefficient $\zeta > 0$. For large frictions $\zeta = \tfrac1\e$, after a proper scaling of time  $\partial_t \mapsto \e \partial_t$, 
\eqref{funsys} is rewritten as 
\begin{equation}
\label{funsys_scaled}
\begin{cases}
\displaystyle{\frac{\del \rho}{\del t} + \frac1\e\dx (\rho u) = 0
} 
&
\\ &
\\[-5pt]
\displaystyle{\rho \frac{\partial u}{\partial t} + \frac1\e \rho u \cdot \nabla_x u= -\frac{1}{\e^2}\rho u -  \frac1\e \rho \nabla_x \frac{\delta \cE}{\delta \rho}   }\, , &
\end{cases}
\end{equation}
or, in terms of   $(\rho, m = \rho u)$
\begin{equation}\label{eq:frictapprox}
 \begin{cases}
	\displaystyle{\rho_{t} +\frac1\e\dx m =0}&   \\
	& \\[-5pt]
       \displaystyle{ m_{t} +\frac1\e \dx \frac{m\otimes m}{\rho} 
       	= -\frac{1}{\e^{2}} m - \frac1\e  \rho\nabla_x \frac{\delta \mathcal{E}(\rho)}{\delta \rho }} \, . &
    \end{cases}
\end{equation}
Note that \eqref{eq:frictapprox} is in conservation form except for the term $\rho \nabla_x \frac{\delta \cE}{\delta \rho}$.
Nevertheless, for all examples treated in this paper, we have
\begin{equation}
\label{stresshyp}
- \rho \nabla_x \frac{\delta \cE}{\delta \rho} =  \nabla_x \cdot S \, ,
\end{equation}
where $ S =S (\rho)$ will be a tensor-valued functional on $\rho$ that plays the role of a stress tensor with components $S_{i j} (\rho)$ with 
$i, j = 1, ... , d$. We refer to \cite{GLT15} for a discussion of the ramifications of that property.

As $\e\downarrow 0$ in \eqref{eq:frictapprox},  we formally obtain the gradient flow dynamic 
\begin{equation}\label{eq:gf}
\rho_t - \dx \left(\rho\nabla_x \frac{\delta \mathcal{E}(\rho)}{\delta \rho }\right ) =0\, .
\end{equation}
The objective of this section is to describe this large friction limit using the relative energy identities
induced by the  functional framework.
Particular examples will  include various interesting systems (see the examples in \cite{GLT15}) and in particular:
\begin{enumerate}
\item the porous medium equation as limit of the Euler equation with friction   \cite{LT12a} corresponds to the choice
\begin{equation*}
\mathcal{E}(\rho) = \int h(\rho) dx\, ;
\end{equation*}

\item the Keller-Segel system as limit of the  Euler-Poisson system with friction (in the case of  attractive potentials), considered in Section \ref{sec:KS},
 is given by the functional 
\begin{equation*}
\mathcal{E}(\rho) = \int \left ( h(\rho) - \tfrac{1}{2}\X\rho c\right )dx\, ,
\end{equation*}
where $\X>0$ and $c$ is viewed as a constraint in terms of the relation 
\begin{equation*}
- \triangle_{x} c + \beta c = \rho - <\rho>, \quad <\rho> = \int \rho dx,\ \beta\geq 0\, ;
\end{equation*}

\item the  Cahn-Hilliard  equation as limit of the the Euler-Korteweg  system with friction  corresponds to the choice
\begin{equation*}
\mathcal{E}(\rho) = \int \left (h(\rho) +\tfrac12 C_\kappa |\nabla_x \rho |^2 \right ) dx, \quad C_\kappa>0\, 
\end{equation*}
and is  investigated in Section \ref{sec:korteweg}.

\end{enumerate} 

\subsection{The energy equation}\label{subset:tenergy}
We start by reviewing and adapting to the relaxation framework certain results from \cite{GLT15}. 
First,  we derive the energy estimate for \eqref{funsys_scaled} or \eqref{eq:frictapprox} in the functional setting.
We assume that the directional derivative (Gateaux derivative) of the functional 
$\cE$ defined by 
\begin{equation*}
d\cE (\rho ; \psi ) = \lim_{\tau \to 0}  \frac{\cE (\rho + \tau \psi) - \cE (\rho)}{\tau} = \frac{d}{d\tau} \cE (\rho + \tau \psi ) \Big |_{\tau = 0}
\end{equation*}
is linear in $\psi$ and can be 
represented via a duality bracket
\begin{equation}
\label{funchyp1}
d\cE (\rho ; \psi ) = \frac{d}{d\tau} \cE (\rho + \tau \psi ) \Big |_{\tau = 0} 
= \big \langle \frac{\delta \cE}{\delta \rho }(\rho) , \psi  \big \rangle\, ,
\end{equation}
with $\frac{\delta \cE}{\delta \rho }(\rho)$ standing for the generator of the bracket. This
property is always satisfied for Frechet differentiable functionals.
Using \eqref{stresshyp}, the potential energy is computed via
\begin{equation}
\label{pote}
\frac{d}{dt} \cE (\rho) =  \langle \frac{\delta \cE}{\delta \rho} (\rho) , \rho_t \rangle 
=
- \frac1\e \langle \frac{\delta \cE}{\delta \rho} (\rho) , \dx (\rho u) \rangle
=
   \frac1\e\int S  : \nabla_x u \, dx\, .
\end{equation}
Now, using again \eqref{stresshyp} and   the momentum equation  \eqref{eq:frictapprox}$_2$ with the standard multiplier $u$, we obtain the usual kinetic energy relation
 \begin{align*}
 \frac{1}{2}\frac{d}{dt} \int \rho |u|^2 dx &= -   \frac{1}{\e^2}\int \rho |u|^2 dx  -  \frac1\e\int  \rho u \cdot \nabla_x \frac{\delta \mathcal{E}(\rho)}{\delta \rho }dx \\
 &  = -  \frac{1}{\e^2}\int \rho |u|^2 dx - \frac1\e\int    S  : \nabla_x u \, dx\, ,
\end{align*}
which, added to \eqref{pote}, finally leads to the standard energy relation
 \begin{equation}\label{eq:energy}
\frac{d}{dt} \left ( \mathcal{E}(\rho) +  \frac{1}{2}\int \rho |u|^2 dx  \right )+\frac{1}{\e^2}\int \rho |u|^2dx =0\, .
\end{equation}

\subsection{Relative energy identity for the relaxation theory}\label{subset:twohypsol}
Next,  we compare two different solutions $(\rho, m)$, $(\bar \rho, \bar m)$ of \eqref{eq:frictapprox} using the relative entropy framework.
To this end, we define also
 the second variation of the functional $\cE(\rho)$ via
\begin{equation*}
d^2 \cE (\rho;  \psi, \varphi) = \lim_{\tau \to 0} 
\frac{ \big \langle \frac{\delta \cE}{\delta \rho }(\rho + \tau \varphi) , \psi  \big \rangle - \big \langle \frac{\delta \cE}{\delta \rho }(\rho) , \psi  \big \rangle }{\tau}
\end{equation*}
(whenever the limit exists), and we assume that this can be represented as a bilinear functional in the form
\begin{equation}
\label{funchyp2}
d^2 \cE (\rho;  \psi, \varphi) = \lim_{\tau \to 0}  
\frac{ \big \langle \frac{\delta \cE}{\delta \rho }(\rho + \tau \varphi) , \psi  \big \rangle - \big \langle \frac{\delta \cE}{\delta \rho }(\rho) , \psi  \big \rangle }{\tau}
= \left \langle  \left \langle \frac{\delta^2 \cE }{\delta \rho^2} (\rho) , ( \psi, \varphi ) \right \rangle \right \rangle\, .
\end{equation}
Moreover, in analogy to \eqref{funchyp1}, we assume that the directional derivative of $S (\rho) $ is expressed as 
a linear functional via a duality bracket,
\begin{equation}
\label{funchyp3}
d S  (\rho ; \psi ) = \frac{d}{d\tau} S (\rho + \tau \psi ) \Big |_{\tau = 0} 
= \big \langle \frac{\delta S}{\delta \rho }(\rho) , \psi  \big \rangle \, ,
\end{equation}
in terms of the generator $\frac{\delta S}{\delta \rho }(\rho)$.
\subsubsection{The relative potential energy} \label{sec-funcrel-pote}
Define the  relative potential energy,
\begin{equation}
\label{relendef}
\cE (\rho | \brho) := \cE(\rho) - \cE (\brho) - \big \langle \frac{\delta \cE}{\delta \rho} (\brho) , \rho - \brho \big \rangle\, ,
\end{equation}
as the quadratic part of the Taylor series expansion of the functional $\cE$ with respect to a reference solution $\bar\rho(x,t)$. If $\cE (\rho)$ is convex,
this quantity  can serve as a measure of distance between the two solutions $\rho$ and $\bar\rho$.

Consider next the weak form of \eqref{stresshyp},
\begin{equation*}
\big \langle \frac{\delta \cE}{\delta \rho} (\rho) , \frac{\del}{\del x_j} (\rho \varphi_j) \big \rangle = - \int S_{i j} (\rho) \, \frac{\del \varphi_i}{\del x_j } \, dx\, .
\end{equation*}
This relation is viewed as a functional in $\rho$;  talking its directional derivative along a direction $\psi$,  
with $\psi$ a smooth test function, we obtain
\begin{equation*}
\begin{aligned}
 \left \langle  \left \langle \frac{\delta^2 \cE }{\delta \rho^2} (\rho) , \big ( \psi, \frac{\del}{\del x_j} (\rho \varphi_j) \big ) \right \rangle \right \rangle
& + 
  \left \langle  \frac{\delta \cE}{\delta \rho} (\rho) , \frac{\del}{\del x_j} (\psi \varphi_j) \right \rangle 
  \\
& =
 - \int  \left \langle  \frac{\delta S_{ij} }{\delta \rho} (\rho) , \psi  \right \rangle   \frac{\del \varphi_i }{\del x_j} \, dx \, . 
 \end{aligned}
 \end{equation*}
The two relations lead to (see \cite[Section 2.1]{GLT15} for the details of this computation):
\begin{equation}
\label{relpote}
\frac{d}{dt} \cE (\rho | \brho) =  \frac{1}{\e}  \int S_{i j} (\rho | \brho )  \frac{\del \baru_i }{\del x_j} \, dx
- \frac1\e \big \langle \frac{\delta \cE}{\delta \rho} (\rho) - \frac{\delta \cE}{\delta \rho} (\brho) , \dx \big (\rho (u - \baru)\big ) \big \rangle \, ,
\end{equation}
where $S ( \rho | \brho)$ stands for the relative stress tensor:
\begin{equation}
\label{relstress}
S ( \rho | \brho) := S(\rho) - S (\brho) - \big \langle \frac{\delta S}{\delta \rho} (\brho) , \rho - \brho \big \rangle\, .
\end{equation}

\subsubsection{The relative kinetic energy} \label{sec-funcrel-kine}
Next consider the kinetic energy 
\begin{equation}
\label{kinedef}
K (\rho, m) = \int \frac{1}{2} \frac{ |m|^2}{\rho} dx
\end{equation}
viewed as a (not strictly) convex functional on the density $\rho$ and the momentum $m = \rho u$. 
The relative kinetic energy is expressed in the form
\begin{equation}
\label{relkinenergy}
\begin{aligned}
K ( \rho, m | \brho, \barm ) &: = \int k(\rho, m) - k(\brho, \barm) - \nabla k (\brho, \barm) \cdot  (\rho - \brho, m - \barm) \, dx
\\
&= \int \frac{1}{2} \frac{ |m|^2}{\rho}  - \frac{1}{2} \frac{ |\barm|^2}{\brho} 
- \big ( -\frac{1}{2} \frac{|\barm |^2}{\brho^2} , \frac{\barm}{\brho} \big ) \cdot (\rho - \brho, m - \barm) \,  dx
\\
&=  \frac{1}{2} \int \rho | u - \baru|^2 \, dx\, ,
\end{aligned}
\end{equation}
To compute its evolution, consider the difference of the two equations satisfied by $(\rho, u)$ and $(\brho, \baru)$, that is
\begin{equation*}
\begin{aligned}
 \partial_t (u-\bar u) &+  \frac{1}{\e}(u\cdot\nabla_x)(u- \bar u) +  \frac1\e\big((u-\bar u)\cdot\nabla_x \big)\bar u 
 \\
 &= - \frac{1}{\e^{2}}  (  u - \bar u ) - \frac1\e     \nabla_x \left (\frac{\delta \mathcal{E}(  \rho)}{\delta \rho}
 - \frac{\delta \mathcal{E}(\bar \rho)}{\delta \rho }\right)\, .
 \end{aligned}
 \end{equation*}
Multiplying this relation by $u-\bar u$ we end up with
\begin{align*}
& \frac{1}{2}  \partial_t  |u-\bar u|^2+ \frac{1}{2\e}(u\cdot\nabla_x)|u- \bar u|^2 +
 \frac1\e \nabla_x \baru : (u - \baru) \otimes (u - \bar u) 
 \\
 &\quad = 
- \frac{1}{\e^2}|u - \bar u|^2   - \frac1\e(u - \bar u)\cdot   \nabla_x \left (\frac{\delta \mathcal{E}(  \rho)}{\delta \rho}
 - \frac{\delta \mathcal{E}(\bar \rho)}{\delta \rho }\right)\, ,
\end{align*}
which, using \eqref{eq:frictapprox}$_1$ and integrating over space leads to the balance of the  relative kinetic energy  
%
\begin{equation}
\label{relkine}
\begin{aligned}
&\frac{1}{2}\frac{d}{dt} \int  \rho |u - \baru|^2 \, dx +  \frac{1}{\e^2} \int \rho |u - \baru|^2 \, dx  =
\\
& - \frac1\e \int  \rho \nabla_x \baru : (u - \baru) \otimes (u - \bar u) \, dx 
+ \frac1\e \big \langle \frac{\delta \cE}{\delta \rho} (\rho) - \frac{\delta \cE}{\delta \rho} (\brho) , \dx\big( \rho (u - \baru) \big)\big \rangle  \, .
\end{aligned}
\end{equation}

\subsubsection{The functional form of the relative energy formula} \label{sec-funcrel-tote}
Summing  \eqref{relpote} to \eqref{relkine} we obtain the relative energy identity
\begin{equation}
\label{reltote}
\begin{aligned}
&\frac{d}{dt} \left ( \cE (\rho | \brho) + \frac12 \int   \rho |u - \baru|^2 \, dx \right ) +  \frac{1}{\e^2}   \int  \rho |u - \baru|^2 \, dx
\\
&\qquad \qquad =  \frac{1}{\e} \int  \nabla_x \baru : S (\rho | \brho )   \, dx  -  \frac{1}{\e} \int  \rho \nabla_x \baru : (u - \baru) \otimes (u - \bar u) \, dx \, ,
\end{aligned}
\end{equation}
where $\cE (\rho | \brho)$ and $S(\rho | \brho)$ stand for the relative potential energy and relative stress functionals 
defined in \eqref{relendef} and \eqref{relstress}, respectively.
The main property which leads to the above relation is the fact that 
 the contributions of the term
\begin{equation*}
D =  \frac{1}{\e} \big \langle \frac{\delta \cE}{\delta \rho} (\rho) - \frac{\delta \cE}{\delta \rho} (\brho) , \dx \big ( \rho (u - \baru)\big) \big \rangle 
\end{equation*}
in \eqref{relpote} and \eqref{relkine} offset each other, as for the terms involving the stress tensor $S$ in the derivation of the energy relation \eqref{eq:energy}.

\subsection{Confinement potentials}\label{sec:confinement}
It is expedient to give an extension of the calculation for systems driven by a confinement potential $V = V(x)$,
\begin{equation}
\label{funsys_conf}
\begin{cases}
\displaystyle{ \; \, \frac{\del \rho}{\del t} + \tfrac1\e\dx (\rho u) = 0
} 
&
\\[-5pt] &
\\
\displaystyle{\rho \frac{\partial u}{\partial t} + \tfrac1\e \rho u \cdot \nabla_x u= -\tfrac{1}{\e^2}\rho u -  \tfrac1\e \rho \nabla_x \frac{\delta \cE}{\delta \rho}   }
- \tfrac1\e \rho \nabla_x  V \, . &
\end{cases}
\end{equation}
The potential energy is now given by the functional
\begin{equation}
\label{functconf}
\cF (\rho) = \cE (\rho) + \int \rho V(x) \, dx
\end{equation}
and we require that $\cE (\rho)$ satisfies \eqref{stresshyp} for some stress functional $S(\rho)$. 
Note, that the potential energy functional splits into the potential energy of the contact forces
$\cE(\rho)$ and the potential energy of the body forces $\int \rho V$. The latter is not expected to be associated to a stress, and also is not invariant under
space translations (which is connected to  the hypothesis \eqref{stresshyp}). Under the hypothesis \eqref{stresshyp} for $\cE$, we may write the
weak form of \eqref{funsys_conf} for  $(\rho, m = \rho u)$
\begin{equation}\label{confin-weak}
 \begin{cases}
	\displaystyle{\rho_{t} +\tfrac1\e\dx m =0}&   \\
	& \\[-5pt]
       \displaystyle{ m_{t} +\tfrac1\e \dx \frac{m\otimes m}{\rho} 
       	= -\tfrac{1}{\e^{2}} m - \tfrac1\e  \nabla_x \cdot S (\rho) - \tfrac1\e \rho \nabla_x  V }\, . &
    \end{cases}
\end{equation}

Proceeding along the lines of the calculations in Section \ref{subset:tenergy} we see that
\begin{align}
\label{confinepote}
\frac{d}{dt} \cF (\rho) &=  \frac1\e\int S  : \nabla_x u \, dx  + \frac1\e\int \rho \nabla_x V \cdot u  \, dx \, ,
\\
\label{confinekine}
 \frac{1}{2}\frac{d}{dt} \int \rho |u|^2 dx &= -  \frac{1}{\e^2}\int \rho |u|^2 dx - \frac1\e\int    S  : \nabla_x u  \, dx - \frac1\e\int \rho \nabla_x V \cdot u  \, dx \, ,
\end{align}
and the total energy for \eqref{funsys_conf} reads
\begin{equation}\label{confinenergy}
\frac{d}{dt} \left ( \mathcal{F}(\rho) +  \frac{1}{2}\int \rho |u|^2 dx  \right )+\frac{1}{\e^2}\int \rho |u|^2dx =0\, .
\end{equation}

On the other hand, due to the formulas
\begin{align*}
\cF (\rho | \brho) &= \cE (\rho | \brho ) \, ,
\\
\frac{\delta \cF}{\delta \rho} (\rho) - \frac{\delta \cF}{\delta \rho} (\brho) &= \frac{\delta \cE}{\delta \rho} (\rho) - \frac{\delta \cE}{\delta \rho} (\brho)\, ,
\end{align*}
the calculations in Sections \ref{sec-funcrel-pote} and \ref{sec-funcrel-kine} remain essentially unaffected, and the final relative energy formula,
comparing two solutions $(\rho, u)$ and $(\brho, \baru)$ of the system \eqref{funsys_conf} with confinement potential,
takes exactly the same form as \eqref{reltote}:
\begin{equation}
\label{reltoteconfine}
\begin{aligned}
&\frac{d}{dt} \left ( \cE (\rho | \brho) + \frac12 \int   \rho |u - \baru|^2 \, dx \right ) +  \frac{1}{\e^2}   \int  \rho |u - \baru|^2 \, dx
\\
&\qquad \qquad =  \frac{1}{\e} \int  \nabla_x \baru : S (\rho | \brho )   \, dx  -  \frac{1}{\e} \int  \rho \nabla_x \baru : (u - \baru) \otimes (u - \bar u) \, dx \, .
\end{aligned}
\end{equation}

 \subsection{The analysis of the diffusive limit}\label{subset:1hyp1gf}
 We return now to the system \eqref{funsys_scaled} without confinement potential.
 In this section  we aim to compare a solution $(\rho, \rho u)$ of \eqref{eq:frictapprox} with a smooth 
 solution 
 $\bar\rho$ of \eqref{eq:gf}. To this end, we define 
 \begin{equation}
 \label{defbarm}
 \bar m = \bar\rho \bar u = - \e \bar \rho \nabla_x\frac{\delta \mathcal{E}( \bar \rho)}{\delta \rho}
 \end{equation} 
and visualize the pair 
 $(\bar \rho, \bar m = \bar \rho\bar u)$ as an approximate solution of  \eqref{eq:frictapprox}, that is 
\begin{equation}\label{eq:gfrewri}
 \begin{cases}
	\displaystyle{\bar\rho_{t} +\frac1\e \dx \bar m =0}&   \\
	& \\
       \displaystyle{ \bar m_{t} + \frac1\e \dx \frac{\bar m\otimes \bar m}{\bar \rho} 
       	= -\frac{1}{\e^{2}}  \bar m - \frac1\e \bar  \rho\nabla_x \frac{\delta \mathcal{E}(\bar \rho)}{\delta \rho } + \bar e}\, ,  &
    \end{cases}
\end{equation}
where 
 \begin{equation*}
\bar e = \bar m_{t} + \frac1\e \dx \frac{\bar m\otimes \bar m}{\bar \rho}\,  .
\end{equation*}
Using \eqref{defbarm} and the smoothness of $\bar \rho$ we see that $\bar e$ is a $O(\e)$ error term.
The only difference from the calculations  of section \ref{subset:twohypsol} lies in the relative kinetic energy \eqref{relkine}. Presently,
$\bar u$ satisfies the approximate equation
 \begin{equation*}
 \bar u_t  +\frac1\e(\bar u\cdot\nabla_x)\bar u =-\frac{1}{\e^2}\bar u -   \frac1\e  \nabla_x \frac{\delta \mathcal{E}(\bar \rho)}{\delta \rho } +  \frac{\bar e}{\bar \rho}\, ,
 \end{equation*}
 and, following the analysis in section \ref{sec-funcrel-kine}, we obtain
 \begin{align} 
 \label{eq:relkinenrelax}
& \frac{1}{2} \partial_t   |u-\bar u|^2+ \frac{1}{2\e}(u\cdot\nabla_x)  |u- \bar u|^2  +
\frac1\e \nabla_x \baru : (u - \baru) \otimes (u - \bar u) 
 \nonumber\\
 &\quad = 
- \frac{1}{2\e^2}|u - \bar u|^2   - \frac1\e (u - \bar u)\cdot   \nabla_x \left (\frac{\delta \mathcal{E}(  \rho)}{\delta \rho}
 - \frac{\delta \mathcal{E}(\bar \rho)}{\delta \rho }\right) -  (u-\bar u)\cdot \frac{\bar e}{\bar \rho}\, ,
\end{align}
 and the relative energy relation 
 \begin{align}
\label{eq:relengfrelax}
&\frac{d}{dt} \left ( \cE (\rho | \brho) + \frac12 \int  \rho |u - \baru|^2 \, dx \right ) + \frac{1}{\e^2}  \int  \rho |u - \baru|^2 \, dx
\nonumber\\
&\quad  = \frac{1}{\e} \int  \nabla_x \baru : S (\rho | \brho )   \, dx  -  \frac{1}{\e} \int   \rho \nabla_x \baru : (u - \baru) \otimes (u - \bar u) \, dx
-  \int  \rho(u-\bar u)\cdot \frac{\bar e}{\bar \rho} dx  \, .
\end{align}
Since $\bar u = O(\e)$,  for smooth solutions the coefficients of the quadratic terms are $O(1)$ in $\e$. 
Moreover, the last (error) term at the right hand side of \eqref{eq:relengfrelax} is controlled in terms of  the distance w.r.t.\ the equilibrium, 
i.e.\ by $\tfrac12 \e^{-2}\rho |u - \bar u|^2$, and an $O(\e^4)$ term depending on total mass of $\rho$ and the smooth, strictly positive solution $\bar\rho$ of \eqref{eq:gf}.
This relation is therefore instrumental to control the relaxation limit.

\subsection{Relative energy estimate for the gradient flow}\label{subset:limit only}
 The above calculations induce a relative energy estimate for comparing two solutions $\rho$ and $\bar \rho$ of the limiting gradient flow \eqref{eq:gf}. 
 Indeed, using  in \eqref{reltote} the expressions \eqref{defbarm} for the two velocities  $u$ and $\bar u$ at equilibrium  we obtain
 \begin{equation}
 \label{relpotegrad}
\frac{d}{dt} \cE (\rho | \brho) + \int  \rho \left | \nabla_x \Big (\frac{\delta \mathcal{E}( \rho)}{\delta \rho}-\frac{\delta \mathcal{E}(\bar \rho)}{\delta \rho} \Big ) \right |^2  dx
 = - \int S_{i j} (\rho | \brho )  \frac{\del^2 }{\del x_j \del x_i}  \frac{\delta \mathcal{E}( \bar \rho)}{\delta \rho}   dx
  \, .
\end{equation}
Note that in this calculation the effect of the kinetic energy drops out,  and the  derivation of  \eqref{relpotegrad} 
uses only  \eqref{relpote} and the transport equations for $\rho$ and $\bar \rho$;  the term $D$ becomes dissipative 
with the velociy choices 
$$u = - \e \nabla_x \frac{\delta \mathcal{E}( \rho)}{\delta \rho} \quad \text{and} \quad \bar u = - \e \nabla_x\frac{\delta \mathcal{E}( \bar \rho)}{\delta \rho}
$$
leading to the gradient flow.

\section{From the Euler-Poisson system with friction to the Keller-Segel system}\label{sec:KS}
In this section, we shall make precise the functional setting  for the Euler--Poisson system with an attractive potential and friction converging 
to Keller--Segel type models. The Euler-Poisson system reads
\begin{equation}
    \begin{cases}
	\displaystyle{\rho_{t} +\frac{1}{\e}\dx m =0}&   \\
	& \\[-5pt]
       \displaystyle{ m_{t} + \frac{1}{\e}\dx \frac{m\otimes m}{\rho} 
       + \frac{1}{\e}\nabla_{x}p(\rho) 
	= -\frac{1}{\e^{2}}m} + \frac{\X}{\e}\rho \nabla_{x} c &  \\
	& \\[-5pt]
	-\triangle_{x} c + \beta c=  \rho - < \rho >,
    \end{cases}
    \label{eq:KS}
\end{equation}
where $t\in\R$, $x\in\T^{n}$, $n=2,3$ the physically relevant dimensions, $\rho\geq 0$, $c\in\R$, 
$m = \rho u \in\R^{n}$. We assume that the internal energy $h(\rho)$ and the pressure $p(\rho)$ are connected through the 
usual thermodynamic relation
\begin{equation}
\label{hypthermo}
\rho h''(\rho) = p'(\rho) \, , \quad \rho h'(\rho) = p(\rho) + h(\rho) \, \quad \mbox{with $p'(\rho) > 0$}.
\tag{H}
\end{equation}
Moreover, we impose the conditions on the pressure that for some constants $k > 0$ and $A > 0$,
\begin{align}
\label{eq:growthh}
h(\rho) &= \frac{k}{\gamma -1} \rho^\gamma + o(\rho^\gamma) \, , \quad \hbox{as}\ \rho\to +\infty
\tag{A$_1$}
\\
\label{hyppress}
| p'' (\rho ) | &\le A \frac{p'(\rho)}{\rho} \qquad \forall \rho > 0 \, .
\tag{A$_2$}
\end{align}
These conditions are satisfied by the usual $\gamma$--law:  $p(\rho) = k\rho^{\gamma}$ with $\gamma > 1$.
For the constants $\gamma > 1$, $\beta\geq 0$ and $\X>0$ (chemosensitive coefficient)  appearing in  \eqref{eq:KS} we will impose
various smallness/largeness conditions that are precised later.
Finally, the elliptic equation in \eqref{eq:KS} is provided with periodic boundary conditions and it shall be intended for zero mean solutions $c$ in the case $\beta=0$; in that equation, as specified in the previous section, $< \rho >$ stand for the mean of $\rho$.


Formally,  after an appropriate scaling of the moment, at the limit $\e\downarrow 0$ we obtain $m =  \X\rho \nabla_{x}c - \nabla_{x} p(\rho) $ and therefore the formal limit of \eqref{eq:KS} is given by the Keller-Segel type model:
\begin{equation}
    \begin{cases}
	\displaystyle{\rho_{t} +\dx \big (\X\rho \nabla_{x}c - 
	\nabla_{x} p(\rho) \big )  =0}&   \\
	- \triangle_{x} c + \beta c = \rho - <\rho>.
    \end{cases}
    \label{eq:KSlimit}
\end{equation}

\subsection{Preliminaries} \label{sec:prelim}
Standard hyperbolic theory  suggests to employ for  \eqref{eq:KS}
the usual entropy--entropy flux pair 
\begin{equation}
   \eta(\rho,m) =  \frac{1}{2}\frac{|m|^{2}}{\rho} + h(\rho),\quad
  \, , \quad
   q(\rho,m) = \frac{1}{2}m\frac{|m|^{2}}{\rho^{2}} + 
    mh'(\rho) \, , 
\label{defentropypair}
\end{equation}
depicting the mechanical energy and its flux.
We recall that for the particular case of $\gamma$--law gases, $p(\rho) = k \rho^\gamma$, $h$ takes the form
\begin{equation*}
    h(\rho) = 
    \begin{cases}
    \displaystyle{\frac{k}{\gamma-1}\rho^{\gamma} = \frac{1}{\gamma-1}p(\rho)} & \hbox{for}\ \gamma>1; \\
   \displaystyle{ k\rho\log\rho  } & \hbox{for}\ \gamma=1.
    \end{cases}
\end{equation*}
An entropy weak solution of \eqref{eq:KS} satisfies in the sense of distribution the entropy inequality 
\begin{equation}
    \eta(\rho,m)_{t} +\frac{1}{\e}\dx q(\rho, m) \leq  
    -\frac{1}{\e^{2}}
    \nabla_{m}\eta(\rho,m)\cdot m = 
    - \frac{1}{\e^{2}}\frac{|m|^{2}}{\rho} + \frac{\X}{\e}  m \cdot \nabla_{x} c
    \label{eq:entrKS}
\end{equation}
On the other hand, smooth solutions of \eqref{eq:KSlimit} satisfy the entropy identity
\begin{equation}\label{eq:entrKSlimit}
    h(\rho)_{t} + \dx \big ( h'(\rho)(\X\rho\nabla_{x} c - 
    \nabla_{x}p(\rho) )\big)
    = - 
    \frac{| \nabla_{x} p(\rho)|^{2}}{\rho} + \X \nabla_{x}p(\rho)\cdot  
    \nabla_{x} c.
\end{equation}
This form of the energy equations is inadequate to carry out the relative entropy analysis of the forthcoming section.

Next, we present a variant of the energy equation,  inspired by the formal analysis of Section \ref{sec:largefrictvsgradflows}.
To start, the solution of the elliptic equation \eqref{eq:KS}$_3$ can be expressed via convolution with the Green's function,
\begin{equation*}
c (x) = (\mathcal{K} \ast \rho) (x) = \int \mathcal{K}( x - y) \rho (y) \, dy\, ,
\end{equation*}
where $\mathcal{K}$ is a symmetric function. The energy of the Euler-Poisson system takes the form
\begin{equation}
\label{epenergy}
\begin{aligned}
\mathcal{E}(\rho) &= \int \left ( h(\rho) - \tfrac{1}{2}\X\rho c\right )dx  
\\
&= \int h(\rho) \, dx - \tfrac{1}{2} \X \iint \rho (x) \mathcal{K} (x - y) \rho (y) dx dy
\, ,
\end{aligned}
\end{equation}
and the symmetry of $\mathcal{K}$  implies 
\begin{equation}
- \rho \nabla_x \frac{\delta \cE}{\delta \rho }(\rho) = - \rho \nabla_x  ( h'(\rho) - \X c)  \, .
\label{hamflowep}
\end{equation}

Next, we show that for the Euler-Poisson system there is a stress associated with \eqref{hamflowep}. Indeed,
multiplying    \eqref{eq:KS}$_3$ by $\nabla_x c$ we end up with 
\begin{equation*}
\rho \nabla_x c = \nabla_x \Big ( \tfrac{1}{2} |\nabla_x c|^2 + \tfrac{\beta}{2} c^2 + <\rho> c \Big ) - \dx ( \nabla_x c \otimes \nabla_x c )\, ,
\end{equation*}
so that
\begin{align}
&- \rho \nabla_x  ( h'(\rho) - \X c) 
\\
&\quad = \dx \Big ( - \big [ p(\rho) -  \tfrac{1}{2}\X ( {\beta} c^2 + |\nabla_x c|^2) - \X <\rho> c \big ] \mathbb{I} - \X \nabla_x c \otimes \nabla_x c \Big ) \, .
\nonumber
\end{align}
This determines the stress $S$ in \eqref{stresshyp} as
\begin{align}
S &=  - \big [ p(\rho) -  \tfrac{1}{2}\X ( {\beta} c^2 + |\nabla_x c|^2) - \X <\rho> c \big ] \mathbb{I} - \X \nabla_x c \otimes \nabla_x c \, ,
\label{stressep}
\end{align}
where $\I$ is the identity matrix. Note that the pressure has a contribution coming from the mean-field interaction term.
The induction of a pressure from the mean filed interaction can also be see from the energy identity
\begin{equation}
\label{energyh1}
    \int  \rho c dx   = \int   \big (  \beta c^2+ |\nabla_x c|^2\big)dx \, ,
\end{equation}
obtained directly from the elliptic equation  \eqref{eq:KS}$_3$.

Following the general framework of Section \ref{subset:tenergy},  the potential energy satisfies
\begin{align*}
\frac{d}{dt}   \int  \left ( h(\rho) - \tfrac{1}{2}\X\rho c\right )  dx  
&=
- \frac1\e  \int \big (h'(\rho) - \X c \big)   \dx (\rho u) dx \nonumber\\
&=
   \frac1\e\int S  : \nabla_x u \, dx\, ,
\end{align*}
the kinetic energy is 
\begin{align*}
 \frac{1}{2}\frac{d}{dt} \int \frac{|m|^2}{\rho} dx &= -   \frac{1}{\e^2}\int \frac{|m|^2}{\rho} dx 
 -  \frac1\e\int  \rho u \cdot \nabla_x ( h'(\rho) - \X c) dx
 \\
 &  = -  \frac{1}{\e^2}\int \frac{|m|^2}{\rho} dx - \frac1\e\int    S  : \nabla_x u \, dx\, ,
\end{align*}
and the total energy reads 
\begin{equation*}
    \frac{d}{dt} \int   \left (\eta(\rho,m) - \tfrac{1}{2}\X \rho c \right) dx   +  \frac{1}{\e^{2}}\int\frac{|m|^{2}}{\rho} dx = 0\, .
\end{equation*}

In accordance to the usual practice in conservation laws, we will define entropy dissipative solutions $(\rho, m, c)$ of \eqref{eq:KS} as weak solutions
satifysing the weak form of the energy inequality
\begin{equation}
    \frac{d}{dt} \int   \left (\eta(\rho,m) - \tfrac{1}{2}\X \rho c \right) dx   +  \frac{1}{\e^{2}}\int\frac{|m|^{2}}{\rho} dx \le  0\, ,
      \label{eq:entrKS-3}
\end{equation}
in the sense of distributions.
Clearly \eqref{eq:entrKS-3} and \eqref{eq:entrKS} are equivalent (for smooth solutions), however the form \eqref{eq:entrKS-3} together with \eqref{energyh1}
suggest that to control the potential energy is tantamount to obtaining control of the  $H^1$ norm of $c$.
The above estimate is the starting point to obtain the stability estimate in terms of relative entropy and the corresponding analysis of the relaxation limit in
the next section. 

\subsection{Relative energy estimate}\label{subsec:relenKS}
In this section we perform a relative energy computation between a weak solution $(\rho, m, c)$  of \eqref{eq:KS} and a strong solution $(\bar \rho, \bar c)$ 
of  \eqref{eq:KSlimit}.
The final formula  \eqref{eq:RelEnKorGenFinalweak} turns out to be a special case of formula \eqref{eq:relengfrelax} derived in Section \ref{sec:largefrictvsgradflows} 
for smooth solutions,
and of the relative energy computation for the Euler--Poisson system  discussed in \cite[Section~2.5]{GLT15}. 
Nevertheless, we shall provide here a direct proof of this identity. The reason is twofold: (i) to justify the relative energy estimate among
a weak and a strong solution, (ii) to account for the effect of error terms appropriate for the relaxation limit problem.
We recall the framework of weak solutions we shall refer to.
\begin{definition}  \label{def:wksol}
A function $( \rho,  m, c)$ with  $\rho \in C([0, \infty) ; L^1 ( \T^n )\cap  L^\gamma ( \T^n ))$, $m \in C  \big ( ( [0, \infty) ;  \big ( L^1(\T^n) \big )^n \big )$, $c\in C([0, \infty) ; H^1 ( \T^n ) )$
$\rho \ge 0$ and  $\frac{m \otimes m}{\rho} \in L^1_{loc}  \left ( ( (0, \infty) \times \T^n ) \right )^{n \times n}$is a \emph{dissipative weak periodic
solution} of \eqref{eq:KS} with finite total energy if 
\begin{itemize}
\item $( \rho,  m, c)$ satisfies the weak form of \eqref{eq:KS}; 
\item $( \rho,  m, c)$ satisfies the following integrated form of the energy inequality \eqref{eq:entrKS-3}:
\begin{equation}
 \label{eq:energyintegrated}
 \begin{aligned}
  - \iint  \Big (  \eta(\rho,m) -  \tfrac{1}{2}\X \rho c \Big ) \dot\theta(t) dxdt + \frac{1}{\e^{2}}\iint \frac{|m|^{2}}{\rho}\theta(t)  dxdt 
  \\
 \le   \int \Big (  \eta(\rho,m) -  \tfrac{1}{2}\X \rho c  \Big ) \Big |_{t=0}  \theta(0)dx \, ,
 \\
       \text{for any non-negative}\ \theta  \in W^{1, \infty} [0, \infty) \ \mbox{compactly supported on $[0, \infty)$},
 \end{aligned}
\end{equation}
\item $( \rho,  m)$ satisfies the following bounds, natural within the given framework:
\begin{align}
\sup_{t\in (0,T)} \int   \rho  dx &=M < \infty\, ,
\nonumber
\\
 \sup_{t\in (0,T)} \int 
 \left (  \eta(\rho,m) -  \tfrac{1}{2}\X \rho c \right ) dx
 & < \infty \, .
 \label{hypCauchyK2}
\end{align}
\end{itemize}
\end{definition}
\begin{remark}\label{rem:creg}
The regularity requested in the above definition is the one needed to  rewrite the equation in terms of the divergence of the stress tensor $S$  in \eqref{stressep}, and it is implied by the finite energy condition. This relies on the $L^\gamma$ integrability of $\rho$ and  elliptic regularity estimates for $c$, and it is proved in Section \ref{sec:stability}; 
see Lemma \ref{lem:ellipticreg}. 
Besides the appropriate smallness condition on the  chemosensitive coefficient $\X>0$
(cfr. \, \eqref{eq:convexitylargecond}), we shall require that $\gamma$ lies in the relevant range for which  \eqref{eq:KSlimit} has regular solutions, that is \eqref{eq:gammacond}.

The existence theory for the Euler-Poisson system \eqref{eq:KS} with attractive potential $\X>0$  (treated here) is largely an open problem. 
Repulsive potentials, $\X<0$, offer a subtle stabilizing mechanism leading to global smooth solutions for potential flows with small velocities \cite{GP11}, 
as well as  existence results for global weak entropy solutions in one-space dimension ({\it e.g.}   \cite{HLY09}). 

\end{remark}

Let   $(\bar \rho, \, \bar c)  : (0,T)\times \T^n \to \R^{n+1}$ be a strong  (conservative) periodic solution 
of \eqref{eq:KSlimit} with $\bar\rho\geq \delta > 0$ for some $\delta > 0$, where 
 the regularity ``strong''  refers to the boundedness of all  the derivatives which will appear lately in the relative energy relation.

As in \cite{LT12a,LT12b} and in Section \ref{subset:1hyp1gf}, we rewrite the equilibrium system \eqref{eq:KSlimit} 
in the variables $\bar\rho$, $\bar c$ and
 \begin{equation}
 \label{defbarmep}
\bar m = - \e \big (\nabla_{x} p(\bar \rho) -  \X\bar\rho \nabla_{x}\bar c \big ) = - \e \bar\rho \nabla_{x} \big ( h'(\bar \rho) - \X \bar c \big )
\end{equation}
as follows:
\begin{equation}
    \begin{cases}
	\displaystyle{\bar\rho_{t} +\frac{1}{\e}\dx \bar m =0}&   \\[6pt]
       \displaystyle{ \bar m_{t} + \frac{1}{\e}\dx \frac{\bar m\otimes \bar m}{\bar\rho} 
       + \frac{1}{\e}\nabla_{x}p(\bar\rho) 
	= -\frac{1}{\e^{2}}\bar m} + \frac{\X}{\e}\bar\rho \nabla_{x} \bar c + e(\bar \rho, \bar m)& \\[6pt]
	-\triangle_{x} \bar c + \beta \bar c=  \bar \rho - <\bar\rho>\, ,
    \end{cases}
    \label{eq:KSbar}
\end{equation}
where the term $e(\bar \rho, \bar m)$ is given by
\begin{align}
   \bar e :=  e(\bar \rho, \bar m) &= \frac{1}{\e} \dx \left (\frac{\bar 
       m\otimes \bar m}{ \bar\rho} \right ) +  \bar m_t  \nonumber\\
       &= 
       \e  \dx \Big (\bar\rho
       \nabla_{x} \big ( h'(\bar \rho) - \X \bar c \big )\otimes \nabla_{x} \big ( h'(\bar \rho) - \X \bar c \big )\Big ) 
       -\e  \del_t \big ( \bar\rho \nabla_{x} \big ( h'(\bar \rho) - \X \bar c \big ) \big ) 
       \nonumber\\
       &= O(\e)\, .
       \label{eq:errortermdef}
\end{align}
Then, the equations satisfied by the differences $\rho-\bar\rho$, $m -\bar m$, $c - \bar c$ are given by
\begin{equation}
\label{eq:diff}
\begin{cases}
\displaystyle{( \rho - \bar \rho)_{t} +\frac{1}{\e}  \pxi (m_i - \bar m_i ) = 0 }&
\\[8pt]
\begin{aligned}
( m - \bar m)_{t} &+ \frac{1}{\e} \pxi( f_i (\rho, m) -  f_{i}(\bar \rho, \bar m) )
\\[-4pt]
	&= -\frac{1}{\e^{2}} ( m - \bar m ) + \frac{\X}{\e}\big( \rho \nabla_{x}  c- \bar\rho \nabla_{x} \bar c\big)  - \bar e
\end{aligned}
	& \\[5pt]
	\displaystyle{ -\triangle_{x} (c- \bar c) + \beta(c- \bar c)=  (\rho - \bar \rho)     - < \rho - \bar \rho >      } \, , &
\end{cases}
\end{equation}
where $i=1,\dots,n$, $f_i$ stands for the (vector) of the flux in \eqref{eq:KS},
\begin{equation}
    f_{i}( \rho,  m) =  m_{i}\frac{m}{\rho} + p(\rho) \mathbb{I}_{i}
\label{eq:fluxEuler}
\end{equation}
and $\mathbb{I}_{i}$ is the $i$--th column of the  identity 
matrix.
With this notations, the integrated version of the entropy relation at the limit \eqref{eq:entrKSlimit} becomes
\begin{equation}
     \frac{d}{dt} \int_{\T^n}   \Big (\eta(\bar\rho,\bar m) -  \tfrac{1}{2}\X \bar\rho \bar c \Big )dx  +  \frac{1}{\e^{2}}\int_{\T^n}\frac{|\bar m|^{2}}{\bar\rho} dx   = 
   \int_{\T^n} \bar e\cdot \frac{\bar m}{\rho} dx \, .
      \label{eq:entrKSLimitrewr}
\end{equation}

\begin{theorem}\label{prop:relenKS}
Let   $(\rho,m,c)$ be as in Definition \ref{def:wksol} and let $(\bar\rho,\bar c)$ be a smooth solution  of \eqref{eq:KSlimit}.
Then
\begin{align}
\label{eq:RelEnKorGenFinalweak}
 & \int_{\T^n} \left. \Big [ \eta(\rho, m \left | \bar \rho, \bar m \right.  ) 
- \tfrac{1}{2}\X(\rho - \bar \rho)(c - \bar c)    \Big ]dx \right |_t 
\nonumber\\
&\ \leq
\int_{\T^n}  \left. \Big [ \eta(\rho, m \left | \bar \rho, \bar m \right.  ) 
- \tfrac{1}{2}\X(\rho - \bar \rho)(c - \bar c)    \Big ]dx \right |_0
\nonumber\\
&\ \  -\frac{1}{\e^2}  \iint_{[0,t]\times\T^n}  \rho \left |\frac{m}{\rho} - \frac{\bar m}{\bar \rho}  \right |^2 dxd\tau  - \iint_{[0,t]\times\T^n} e(\bar \rho, \bar m)\cdot \frac{\rho}{\bar\rho}\left ( \frac{m}{\rho} - \frac{\bar m}{\bar \rho} \right )  dxd\tau
\nonumber\\
&\ \ -  \frac{1}{\e}\iint_{[0,t]\times\T^n} \dx \big (\frac{\bar m}{\bar\rho} \big ) 
 p(\rho \left  | \bar \rho \right .) dxd\tau 
\nonumber\\
&\ \ + \frac{\X}{\e}\iint_{[0,t]\times\T^n} \dx \big (\frac{\bar m}{\bar\rho} \big ) 
 \big ( \tfrac{\beta}{2}  (c - \bar c)^2 +  \tfrac{1}{2} |\nabla_x (c-\bar c)|^2 + (c - \bar c) < \rho - \bar \rho >  \big)  dxd\tau 
\nonumber \\
& \ \  - \frac{\X}{\e}\iint_{[0,t]\times\T^n} \nabla_x  \big (\frac{\bar m}{\bar\rho} \big ) :  \nabla_x (c - \bar c) \otimes \nabla_x (c - \bar c) dxd\tau
\nonumber\\
&\ \ - \frac{1}{\e} \iint_{[0,t]\times\T^n} \rho \nabla_x  \left (\frac{\bar m}{\bar\rho} \right ) : \left (\frac{m}{\rho} - \frac{\bar m}{\bar\rho} \right ) \otimes \left (\frac{m}{\rho} - \frac{\bar m}{\bar\rho} \right )dxd\tau\, ,
\end{align}
where 
\begin{equation}
  \eta(\rho, m \left | \bar \rho, \bar m \right. )  =  h(\rho | \bar \rho)  + \tfrac{1}{2} \rho \big | \frac{m}{\rho} - \frac{\bar m}{\bar \rho} \big |^2  \, .
 \end{equation}
\end{theorem}

\begin{proof}
Let $(\rho, m,c)$ be a weak solution according to Definition \ref{def:wksol} and $(\brho , \bar m, \bar c)$ a strong 
 solution of \eqref{eq:KSbar}. 
 We introduce in \eqref{eq:energyintegrated} the standard choice of test function
\begin{equation}
\theta(\tau) := 
\begin{cases}
1, &\hbox{for}\ 0\leq \tau < t, \\
\frac{t-\tau}{\mu} + 1, &\hbox{for}\ t\leq \tau < t+\mu , \\
0, &\hbox{for}\ \tau \geq t+\mu ,
\end{cases}
\label{testthetaS}
\end{equation}
and let  $\mu \downarrow 0$; we then obtain 
\begin{equation}
\label{lem:ret1}
  \left. \int_{\mathbb{T}^n}\big (  \eta(\rho,m)  -  \tfrac{1}{2}\X \rho c \big)dx  \right |^{t}_{\tau=0}
  \leq
- \frac{1}{\e^2}\iint_{[0,t]\times\T^n} \frac{|m|^2}{\rho}dx d\tau \, .
\end{equation}
Moreover,  time integration of \eqref{eq:entrKSLimitrewr} gives
\begin{equation}\label{lem:ret}
\begin{aligned}
&\left.  \int_{\T^n}   \Big (\eta(\bar\rho,\bar m) -  \tfrac{1}{2}\X \bar\rho \bar c \Big )dx \right |^{t}_{\tau=0}  
\\
&\quad = -  \frac{1}{\e^{2}}\iint_{[0,t]\times\T^n}\frac{|\bar m|^{2}}{\bar\rho} dx d\tau  + 
   \iint_{[0,t]\times\T^n} \frac{\bar m}{\bar\rho}\cdot \bar e  dx d\tau \, .
\end{aligned}
\end{equation}

Next, for the linear part of the relative energy, we consider the  weak formulation 
for the equations satisfied by the differences $(\rho - \bar\rho, m- \bar m)$ in \eqref{eq:diff}:
\begin{align}
&- \iint_{[0,+\infty)\times\T^n} {\psi}_t (\rho - \bar \rho) +  \frac{1}{\e}\psi_{x_i} (m_i - \bar m_i)   dx dt
- \int_{\mathbb{T}^n} \psi (\rho - \bar \rho) \Big |_{t=0} \, dx = 0\, ,
\label{weakmass}
\\
&- \iint_{[0,+\infty)\times\T^n} \left (\varphi_t \cdot (m - \bar m) + \frac{1}{\e}  \partial_{x_i}\varphi_j  \left (\frac{m_i m_j}{\rho} - \frac{\bar m_i \bar m_j}{\bar\rho} \right)\right.
\nonumber
\\
&\quad  \left.+ \frac{1}{\e} \dx\varphi \big( p(\rho) - p(\bar\rho) \big)  \right ) dx dt
 - \int_{\mathbb{T}^n} \varphi \cdot (m - \bar m) \Big |_{t=0}  dx 
\nonumber\\
&\ =  -\frac{1}{\e^2}\iint_{[0,+\infty)\times\T^n} \varphi \cdot ( m - \bar m )dxdt    
\nonumber\\
&\quad + \frac{1}{\e}\iint_{[0,+\infty)\times\T^n} \X\varphi \cdot\big( \rho \nabla_{x}  c- \bar\rho \nabla_{x} \bar c\big)dxdt
-   \iint_{[0,+\infty)\times\T^n} \varphi\cdot \bar e dxdt\, ,
\label{weakmomentum}
\end{align}
where $\varphi$, $\psi$ are Lipschitz test functions 
compactly supported in $ [0,+\infty)$ in time  and periodic in space.
In the above relations we introduce the test functions
$$
\psi = \theta(\tau)\left(h'(\bar \rho)  - \X\bar c - \frac{1}{2}\frac{|\bar m|^2}{\bar\rho^2}   \right ) , \quad
\varphi = \theta(\tau) \frac{\bar m}{\bar\rho} \,  ,
$$
with $\theta(\tau)$ as in \eqref{testthetaS}. This choice  in  \eqref{weakmass}, sending $\mu \downarrow 0$, leads to
\begin{align}
& \int_{\mathbb{T}^n}  \left (  h'(\brho)  -  \X\bar c - \frac{1}{2}\frac{|\bar m|^2}{\bar\rho^2}   \right ) (\rho - \bar \rho) 
 \Bigg |_{\tau=0}^t dx 
\nonumber\\
&\ - \iint_{[0,t]\times\T^n} \partial_\tau\left(  h'(\brho)  - \X\bar c - \frac{1}{2}\frac{|\bar m|^2}{\bar\rho^2}   \right ) (\rho - \bar \rho) 
dx d\tau 
\nonumber\\
&\ - \frac{1}{\e} \iint_{[0,t]\times\T^n} \nabla_x\left(  h'(\brho)  -  \X\bar c - \frac{1}{2}\frac{|\bar m|^2}{\bar\rho^2} \right ) \cdot (m- \bar m)   dx d\tau = 0\, .
\label{weakmass3}
\end{align}
Similarly, from \eqref{weakmomentum} we get
\begin{align}
& \int_{\mathbb{T}^n} \frac{\bar m}{\bar\rho}\cdot (m - \bar m)  \Big |_{\tau=0}^tdx
- \iint_{[0,t]\times\T^n} \partial_\tau\left( \frac{\bar m}{\bar\rho}   \right ) \cdot (m-\bar m) dx d\tau 
\nonumber\\
&\ \ - \frac{1}{\e}
\iint_{[0,t]\times\T^n}\left  (  \partial_{x_i}\big( \frac{\bar m_j}{\bar\rho}   \big ) \Big ( \frac{m_im_j}{\rho}- \frac{\bar m_i\bar m_j}{\bar\rho}  \Big )
    + \dx \left( \frac{\bar m}{\bar\rho}   \right ) \big(p(\rho) - p(\bar\rho)\big)  \right )   dx d\tau 
  \nonumber\\ 
  &\ = -\frac{1}{\e^2}\iint_{[0,t]\times\T^n} \frac{\bar m}{\bar\rho} \cdot ( m - \bar m )dxd\tau  +
  \frac{1}{\e}\iint_{[0,t]\times\T^n} \X  \frac{\bar m}{\bar\rho}\cdot\big( \rho \nabla_{x}  c- \bar\rho \nabla_{x} \bar c\big)dxd\tau
   \nonumber\\ 
   &\ \
   -  \iint_{[0,t]\times\T^n} \frac{\bar m}{\bar\rho}\cdot \bar e dx d\tau\, .
\label{weakmomentum2}
\end{align}
Since $c  = \mathcal{K} \ast \rho $, $\bar c  = \mathcal{K} \ast \bar \rho$ and $\mathcal{K}$ is symmetric, 
\begin{equation*}
 \int_{\mathbb{T}^n} c \bar \rho dx = \int_{\mathbb{T}^n} \bar c  \rho dx
\end{equation*}
and therefore
\begin{equation*}
 \int_{\mathbb{T}^n} \Big ( - \tfrac12 \X  \rho c + \tfrac12 \X \bar\rho\bar c   + \X \bar c(\rho - \bar \rho)\Big ) dx = - \int_{\mathbb{T}^n}  \tfrac{1}{2}\X(\rho - \bar \rho)(c - \bar c) dx\, .
\end{equation*}
Hence, we compute \eqref{lem:ret1} -  \eqref{lem:ret} - \big(\eqref{weakmass3} + \eqref{weakmomentum2}\big) to obtain
\begin{align}
& \int_{\mathbb{T}^n} \Big [ \eta(\rho, m \left | \bar \rho, \bar m \right.  ) 
- \tfrac{1}{2}\X(\rho - \bar \rho)(c - \bar c)    \Big ]\Bigg |_{\tau=0}^t dx 
\nonumber\\
&\ \leq - \frac{1}{\e^2}\iint_{[0,t]\times\T^n} \Big [  \rho|u|^2  - \bar\rho|\bar u|^2   - \bar u \cdot  ( \rho u - \bar \rho\bar u ) \Big ]dx d\tau
\nonumber\\
&\ \ - \iint_{[0,t]\times\T^n} \Big [ \partial_\tau\big(  h'(\brho)  - \X\bar c - \tfrac{1}{2} |\bar u|^2  \big ) (\rho - \bar \rho)  
+ \partial_\tau(\bar u )\cdot (\rho u - \bar\rho \bar u)  \Big ] dxd\tau 
\nonumber\\
&\ \ - \frac{1}{\e} \iint_{[0,t]\times\T^n} \nabla_x\big(  h'(\brho)  -  \X\bar c - \tfrac{1}{2}|\bar u|^2 \big ) \cdot (\rho u- \bar \rho \bar u)   dx d\tau
\nonumber\\
&\ \ - \frac{1}{\e} \iint_{[0,t]\times\T^n} \partial_{x_i} ( \bar u_j )  \Big (  ( \rho u_i u_j -  \bar \rho \bar u_i \bar u_j   )
  +(p(\rho) - p(\bar\rho)) \delta_{ij}     \Big )   dx d\tau
  \nonumber\\ 
 &\ \ - \frac{1}{\e}\iint_{[0,t]\times\T^n} \X  \bar u \cdot\big( \rho \nabla_{x}  c- \bar\rho \nabla_{x} \bar c\big)dxd\tau \, ,
\label{eq:int1}
\end{align}
for $m = \rho u$ and $\bar m = - \e \bar\rho \nabla_{x} \big ( h'(\bar \rho) - \X \bar c \big )= \bar\rho\bar u$.
Since $\bar u$ verifies the equation
\begin{equation*}
\partial_t \bar u + \frac{1}{\e}( \bar u \cdot \nabla_x ) \bar u  = -  \frac{1}{\e^2} \bar u  - \frac{1}{\e}\nabla_x \big ( h'(\bar \rho) - \X \bar c    \big ) + \frac{\bar e}{\bar \rho}\, ,
\end{equation*}
then
\begin{align*}
&\partial_\tau \big (- \tfrac{1}{2} |\bar u|^2  \big )(\rho - \bar \rho) + \partial_\tau(\bar u )\cdot (\rho u - \bar\rho \bar u) + \frac{1}{\e} 
\nabla_x\big( - \tfrac{1}{2}|\bar u|^2 \big) \cdot (\rho u- \bar \rho \bar u)
\\
&\ +\frac{1}{\e}  \partial_{x_i} ( \bar u_j )     ( \rho u_i u_j -  \bar \rho \bar u_i \bar u_j   ) = -\frac{1}{\e^2}\rho \bar u\cdot(u-\bar u) - \frac{1}{\e}\rho \nabla_x \big ( h'(\bar \rho) - \X \bar c    \big )
\cdot(u-\bar u) 
\\
&\ + \bar e \frac{\rho}{\bar\rho}\cdot(u-\bar u)  + \frac{1}{\e} \rho \nabla_x    (\bar u) :   (u - \bar u) \otimes  (u - \bar u)\, .
\end{align*}
Then, using the above relation and the continuity equation for $\bar\rho$, after some straightforward calculations \eqref{eq:int1} becomes
\begin{align}
& \int_{\mathbb{T}^n} \Big [ \eta(\rho, m \left | \bar \rho, \bar m \right.  ) 
- \tfrac{1}{2}\X(\rho - \bar \rho)(c - \bar c)    \Big ]\Bigg |_{\tau=0}^t dx 
\nonumber\\
&\ \leq - \frac{1}{\e^2}\iint_{[0,t]\times\T^n} \rho |u - \bar u|^2dx d\tau - \iint_{[0,t]\times\T^n} \bar e \frac{\rho}{\bar\rho}\cdot(u-\bar u) dx d\tau
\nonumber\\
&\ \ - \frac{1}{\e}\iint_{[0,t]\times\T^n}  \rho \nabla_x    (\bar u) :   (u - \bar u) \otimes  (u - \bar u)dx d\tau 
\nonumber\\
&\ \ - \frac{1}{\e}\iint_{[0,t]\times\T^n} \dx (\bar u) p(\rho \left | \bar \rho \right.  )dx d\tau
\nonumber\\
&\ \ -
 \iint_{[0,t]\times\T^n} \left [\frac{1}{\e} \X \rho \bar u\cdot\nabla_x(c - \bar c) - \X (\rho -\bar\rho)\partial_\tau \bar c  \right ]dx d\tau\, 
\label{eq:int2}
\end{align}
and we are left with the control of the last term in \eqref{eq:int2}.
Now we use again the symmetry of $\mathcal{K}$ and the equation for $\bar \rho$ to conclude
\begin{align*}
 \int_{\mathbb{T}^n} (\rho - \bar\rho) \partial_\tau \bar c  dx & = \int_{\mathbb{T}^n}   (\rho - \bar\rho)  \mathcal{K} \ast \bar \rho_\tau dx = 
 \int_{\mathbb{T}^n}   (  \mathcal{K} \ast \rho -   \mathcal{K} \ast \bar\rho)\bar\rho_\tau dx 
 \\
  &= 
  \int_{\mathbb{T}^n}   (  c-   \bar c)\bar\rho_\tau dx \\
  &= - \frac{1}{\e}  \int_{\mathbb{T}^n}   (  c-   \bar c) \dx (\rho \bar u) dx +  \frac{1}{\e}  \int_{\mathbb{T}^n}   (  c-   \bar c) \dx \big ((\rho - \bar \rho) \bar u \big)dx  \, 
\end{align*}
and the last term of  \eqref{eq:int2} is rewritten as
\begin{equation*}
I = - \frac{\X}{\e} \int_{\mathbb{T}^n} \nabla_x(  c-   \bar c) \cdot \bar u (\rho - \bar \rho) dx  \, .
\end{equation*}
Finally,  \eqref{eq:diff}$_3$ and integration by parts allow to reexpress
$$
\begin{aligned}
I &= - \frac{\X}{\e} \int_{\mathbb{T}^n}  \nabla \bar u : \nabla (c - \bar c) \otimes \nabla (c - \bar c) dx
\\
&\quad+ \frac{\X}{\e} \int_{\mathbb{T}^n}  \dx \bar u \big ( \tfrac{1}{2} |\nabla (c - \bar c)|^2 + \tfrac{\beta}{2} (c - \bar c)^2 + (c - \bar c) < \rho - \bar \rho > \big ) dx
\end{aligned}
$$
and complete the proof of  \eqref{eq:RelEnKorGenFinalweak}.
\end{proof}

\subsection{Stability estimate  and convergence of the relaxation limit}\label{sec:stability}
In this section we  obtain a stability estimate in terms of the relative energy inequality, and 
establish convergence in the diffusive relaxation limit from \eqref{eq:KS} to \eqref{eq:KSlimit}. 
We restrict to pressure functions satisfying \eqref{hypthermo}, \eqref{eq:growthh}, \eqref{hyppress} with
\begin{equation}
\gamma \geq 2 - \frac2n,\ \hbox{if}\ n\ge3;\quad  \gamma > 2 - \frac2n = 1 ,\ \hbox{if}\ n=2\, .
\tag{H$_{exp}$}
\label{eq:gammacond}
\end{equation}
We note that the $\gamma$ threshold    obtained   here is the same obtained  for  the Keller--Segel equilibrium system in the study of the existence of 
 global in time weak solutions (no blow-up for finite time) for general initial data   \cite{Sug06, WCH16}, both being related with  Hardy-Littlewood-Sobolev inequalities.

We start with some preliminary results.

\begin{lemma}\label{lem:generalconvEuler3d}
(a)  Let $h\in C^0[0,+\infty)\cap C^2(0,+\infty)$ satisfy $h''(\rho)>0$ and \eqref{eq:growthh}
for  $\gamma >1$. 
If $\bar\rho\in K =[\delta, \bar R]$ with $\delta>0$ and $\bar R<+\infty$, then there exist positive constants
$R_0$ (depending on $K$) and $C_1$, $C_2$ (depending on $K$ and $R_0$)
such that
\begin{equation}
\label{eq:hnorm}
h(\rho \left | \bar \rho \right.) \geq 
\begin{cases}
C_1 |\rho-\bar\rho|^2, &\hbox{for}\ 0\leq \rho\leq R_0,\ \bar\rho\in K, \\
C_2 |\rho-\bar\rho|^\gamma, &\hbox{for}\ \rho> R_0,\ \bar\rho\in K.
\end{cases}
\end{equation}

\noindent
(b) If $p(\rho)$ and $h(\rho)$ satisfy \eqref{hypthermo} and \eqref{hyppress} then
\begin{equation}
\label{eq:hyppress}
| p(\rho | \bar \rho) | \le A \, h (\rho | \bar \rho) \quad \forall \rho, \bar \rho > 0 \, .
\end{equation}
\end{lemma}

\begin{proof}
Part (a) is proved in \cite[Lemma 2.4]{LT12a}.  To show (b), one first easily checks the identity
$$
\begin{aligned}
p( \rho | \bar \rho)  &= p(\rho) - p ( \bar \rho) - p'(\bar \rho) (\rho - \bar \rho)
\\
&= (\rho - \bar \rho)^2 \int_0^1 \int_0^\tau p'' (s\rho + (1-s) \bar \rho) ds d\tau \, ,
\end{aligned}
$$ 
and a similar identity holds for $h(\rho | \bar \rho)$. Recall  now that $h'' = \frac{p'}{\rho}$. 
Then hypothesis \eqref{hyppress} implies $|p''| \le A h''$ and thus
$$
\begin{aligned}
| p (\rho | \bar \rho ) | &\le (\rho - \bar \rho)^2 \int_0^1 \int_0^\tau |p'' (s\rho + (1-s) \bar \rho) | ds d\tau
\\
&\le A  (\rho - \bar \rho)^2 \int_0^1 \int_0^\tau h'' (s\rho + (1-s) \bar \rho) ds d\tau 
\\
&= A \,  h(\rho | \bar \rho )
\end{aligned}
$$
and \eqref{eq:hnorm} follows.
\end{proof}

\begin{remark}\label{rem:hnorm}
For exponents $\gamma \ge 2$, by enlarging if necessary $R_0$ so that $|\rho - \bar\rho| \geq 1$ for $\rho > R_0$ and $\bar\rho\in K =[\delta, \bar R]$, we   obtain
\begin{equation}\label{eq:hnorm2}
h(\rho \left | \bar \rho \right.) \geq c_0 |\rho-\bar\rho|^2, \quad \hbox{for}\ \gamma \geq2,\ 
\rho\geq0,\ \bar\rho\in K\, ,
\end{equation}
where $c_0>0$ depends solely on $K$.
\end{remark}

Application of standard elliptic theory to the equation
\begin{equation}\label{eq:elliptic}
-\triangle_{x} (c- \bar c) + \beta(c- \bar c)=  (\rho  - \bar \rho) \, -  < \rho - \bar \rho > \, .
\end{equation}
provides another preliminary estimate.

\begin{lemma}\label{lem:ellipticreg}
Let   $(\rho,m,c)$ be as in Definition \ref{def:wksol} and let $(\bar\rho,\bar c)$ be a smooth solution  of \eqref{eq:KSlimit} such that $\bar\rho\geq\rho^*>0$ 
Then, for any $q\in \left [\frac{2n}{n+2} ,+\infty\right )$ for $n \ge 3$, or $q\in \left (1 , \infty\right )$ for $n=2$, one has
\begin{equation}\label{eq:ellipticreg}
 \int_{\T^n} \left (  \beta |c - \bar c|^2dx + |\nabla_x(c - \bar c)|^2 \right ) \, dx 
=\left | \int_{\T^n} (\rho - \bar \rho)(c-\bar c)dx\right | \leq C \| \rho - \bar \rho \|^2 _{L^q(\T^n)} \,  ,
\end{equation}
where $C$ depends only on the space dimension $n$ and on $\T^n$.
\end{lemma}
\begin{proof}
First of all, we note that  $c-\bar c$ is of zero mean in $\T^n$   for $t\geq 0$ and for any $\beta > 0$, and also by definition for $\beta=0$.
We apply elliptic estimates for \eqref{eq:elliptic} for any non negative screening coefficient $\beta$, including $\beta=0$. Indeed,
we multiply this equation by $c - \bar c$ and integrate over $\T^n$ to get
\begin{equation*}
\begin{aligned}
\left | \int_{\T^n} (\rho - \bar \rho)(c-\bar c)dx\right |
&\leq \beta \int_{\T^n} |c - \bar c|^2dx + 
\int_{\T^n} |\nabla_x(c - \bar c)|^2dx
\\
& \leq C \int_{\T^n} |\nabla_x(c - \bar c)|^2dx,
\end{aligned}
\end{equation*}
using Poincar\'e's inequality if $\beta >0$, for a constant $C$ depending only on $\T^n$. 
We use the Sobolev embedding theorem in the form
$$
W^{1,q} (\T^n) \subset L^2 (\T^n)  \quad \mbox{ for $q \ge \tfrac{2 n}{n+2}$}
$$
and then $L^q$ elliptic regularity estimates ($q \ne 1$) to conclude
\begin{equation*}
\begin{aligned}
\left \| \nabla_x(c - \bar c)\right\|_{L^2(\T^n)}
&\leq C \left \| \nabla_x(c - \bar c)\right\|_{W^{1,q}(\T^n)}
\\
&\leq C \left \| c - \bar c\right\|_{W^{2,q}(\T^n)}
\\
& \leq C \left \| \rho  - \bar \rho\right\|_{L^q(\T^n)}
\end{aligned}
\end{equation*}
for any $\frac{2n}{n+2} \le q < \infty$ for $n \ge 3$, restricted to  $ \frac{2n}{n+2} = 1 < q < \infty$  in the case $n=2$.
\end{proof}

\begin{lemma}
Let $\gamma$ satisfy \eqref{eq:gammacond}, let $(\rho, m, c)$ and $(\bar \rho, \bar c)$ be as in Lemma \ref{lem:ellipticreg}.
There exists a constant $K > 0$ such that
\begin{equation}\label{mainbound}
\left | \int_{\T^n} (\rho - \bar \rho)(c-\bar c)dx\right | \leq K \int_{\T^n} h(\rho \left | \bar \rho \right.)dx\, .
\end{equation}
\end{lemma}

\begin{proof}
We work in the range $2 - \frac{2}{n}\leq\gamma$ for $n=3$ and $2 - \frac{2}{n}<\gamma$ for $n=2$ and split the proof in
two cases: $\gamma < 2$ and $\gamma \ge 2$.

{\it Case $\gamma < 2$:}
Note that, for any $n\geq2$, $2 - \frac{2}{n}\geq \frac{2n}{n+2}$. Using \eqref{eq:ellipticreg} 
with $ \frac{2n}{n+2}\leq q <\gamma < 2 $ for  $n=3$, and $ \frac{2n}{n+2}< q <\gamma < 2$ for  $n=2$, we obtain
\begin{align}
\left | \int_{\T^n} (\rho - \bar \rho)(c-\bar c)dx\right | 
&\leq C  \left (  \int_{\T^n\cap\{\rho\leq R_0\}} | \rho - \bar\rho|^q dx \right )^{\frac{2}{q}}
 + 
 C \left (  \int_{\T^n\cap\{\rho> R_0\}} | \rho - \bar\rho|^qdx \right)^{\frac{2}{q}}
 \nonumber
\\
\label{eq:gammasmall1}
 &\leq \bar{C}_1(\T^n )  \int_{\T^n\cap\{\rho\leq R_0\}}  | \rho - \bar\rho|^2 dx 
 + 
 C \left (  \int_{\T^n\cap\{\rho> R_0\}} | \rho - \bar\rho|^qdx \right)^{\frac{2}{q}}, 
\end{align}

For $1<q<\gamma$, the second term on the right--hand--side of \eqref{eq:gammasmall1} is treated via interpolation between the $L^1$  and $L^\gamma$ spaces:
For $\theta$ satisfying $\frac{1}{q} = \frac{\theta}{\gamma} + 1-\theta$, it is 
\begin{equation}\label{eq:gammasmall2}
\| \rho -\bar\rho\|^2_{L^q({\T^n\cap\{\rho> R_0\})}} \leq \| \rho -\bar\rho\|^{2-2\theta}_{L^1({\T^n\cap\{\rho> R_0\})}} \| \rho -\bar\rho\|^{2\theta}_{L^\gamma({\T^n\cap\{\rho> R_0\})}}, 
\end{equation}
We select $2\theta=\gamma$ and the corresponding $q$ becomes $q=\frac{2}{3-\gamma}$; this
$q$  is indeed admissible (i.e.\ 
$1<\frac{2n}{n+2}\leq q <\gamma$ if $n=3$; $ 1=\frac{2n}{n+2}< q <\gamma$ if $n=2$), in view of  \eqref{eq:gammacond}.
Hence, \eqref{eq:gammasmall2}, the conservation of mass, and \eqref{eq:hnorm} give
\begin{align}
\left (  \int_{\T^n\cap\{\rho> R_0\}} | \rho - \bar\rho|^qdx \right)^{\frac{2}{q}}  &\leq  C (M + \bar M)^{2-\gamma}  \int_{\T^n\cap\{\rho> R_0\}} | \rho - \bar\rho|^\gamma dx
\nonumber
\\
&\leq  \frac{C}{C_2} (M + \bar M)^{2-\gamma} \int_{\T^n\cap\{\rho> R_0\}} h(\rho \left | \bar \rho \right.)dx\, .
\label{relat1}
\end{align}
Combining \eqref{eq:gammasmall1}, \eqref{relat1} and \eqref{eq:hnorm} we derive \eqref{mainbound} for $\gamma < 2$.

{\it Case $\gamma\geq2$:} 
In that case, we select  $q = 2 > \frac{2n}{n+2}$, and use  \eqref{eq:hnorm2} and \eqref{eq:ellipticreg} to immediately obtain
\begin{equation}\label{eq:gamma>2}
\left | \int_{\T^n} (\rho - \bar \rho)(c-\bar c)dx\right | \leq \frac{C}{c_0} \int_{\T^n} h(\rho \left | \bar \rho \right.)dx\, , 
\end{equation}
where $C$, $c_0>0$ depend solely on  the dimension, $\T^n$ and on the bounds of the smooth limit solution  $\bar\rho$.
\end{proof}

Our objective is to use the relative energy of the Euler-Poisson system
\begin{equation}
\Phi(t) = \int_{\T^n} \left ( \tfrac{1}{2} \rho \left | \frac{m}{\rho} - \frac{\bar m}{\bar \rho} \right |^2 
+  h(\rho | \bar \rho)   -  \tfrac{1}{2} \X  (\rho - \bar \rho ) ( c - \bar c)   \right )\, dx 
\label{eq:norm}
\end{equation}
as a yardstick for estimating the error terms in the relative entropy identity \eqref{eq:RelEnKorGenFinalweak}. To ensure
that \eqref{eq:norm} will measure distance of solutions, we restrict the values of the chemosensitive coefficient $\X$:

\begin{proposition}\label{prop:convHlarge}
Let   $(\rho,m,c)$ be as in Definition \ref{def:wksol} and let $(\bar\rho,\bar c)$ be a smooth solution  of \eqref{eq:KSlimit} such that $\bar\rho\geq\rho^*>0$. 
If the parameter
\begin{equation}\label{eq:convexitylargecond}
 {\X}< \frac{2}{K} \, , 
 \tag{H$_c$}
\end{equation} 
where $K$ is defined in \eqref{mainbound}, then for $\lambda := 1-  \tfrac{1}{2} K \X > 0$ the relative energy $\Phi (t)$ satisfies
\begin{equation}\label{eq:convexityresult}
\Phi (t)  \ge 
 \int_{\T^n} \left ( \tfrac{1}{2} \rho \left | \frac{m}{\rho} - \frac{\bar m}{\bar \rho} \right |^2 
+ \lambda  h(\rho | \bar \rho)  \right ) \, dx  > 0  \, .
\end{equation}
\end{proposition}

We prove: 

\begin{theorem}\label{th:finalKS}
Let  $T>0$ be fixed,  let   $(\rho,m,c)$ be as in Definition \ref{def:wksol} and  $(\bar\rho,\bar c)$ be a smooth solution  of \eqref{eq:KSlimit} such that $\bar\rho\geq \delta >0$.
Assume that the pressure $p(\rho)$ satisfies \eqref{hypthermo}, \eqref{eq:growthh}, \eqref{hyppress}, with exponent $\gamma$ restricted by \eqref{eq:gammacond},
the coefficient $\X$ is restricted by \eqref{eq:convexitylargecond}, and $\beta \ge 0$. Then, the  stability estimate
\begin{equation*}
\Phi(t) \leq C \big (\Phi(0) + \e^4 \big ), \quad t\in [0,T] \, ,
\end{equation*}
holds true,  where $C$ is a positive constant depending on $T$, $\X$, $M$, $\bar\rho$ and its derivatives.
Moreover, if $\Phi(0)\to 0$ as $\e \downarrow 0$, then
\begin{equation*}
\sup_{t\in[0,T]}\Phi(t) \to 0, \ \hbox{as}\ \e \downarrow 0\, .
\end{equation*}
\end{theorem}
\begin{proof}
We employ \eqref{eq:RelEnKorGenFinalweak} for the relative energy  \eqref{eq:norm} to conclude 
\begin{equation}
\label{last1}
  \Phi(t)  + \frac{1}{\e^2}  \int_0^t\int_{\T^n}  \rho \left |\frac{m}{\rho} - \frac{\bar m}{\bar \rho}  \right |^2 dxd\tau  \leq \Phi(0) + \int_0^t\int_{\T^n}\big(|Q| + |E|\big) dxd\tau\, ,
\end{equation}
for $t\in[0,T]$. The quadratic $Q$ and error $E$ terms in \eqref{last1} are defined by
\begin{align*}
E&:= \bar e \cdot \frac{\rho}{\bar\rho}\left ( \frac{m}{\rho} - \frac{\bar m}{\bar \rho} \right )\; \\
Q&:= \frac{1}{\e} 
\dx \left (\frac{\bar m}{\bar\rho} \right ) \Big [  - p(\rho \left  | \bar \rho \right .) +  \X \big (\tfrac{1}{2}\beta (c - \bar c)^2 +\tfrac{1}{2} |\nabla_x (c-\bar c)|^2 + (c - \bar c) < \rho - \bar \rho > \big) \Big ] 
 \\
& \   - \frac{\X}{\e}\nabla_x  \left (\frac{\bar m}{\bar\rho} \right ) :  \nabla_x (c - \bar c) \otimes \nabla_x (c - \bar c) 
 -  \frac{1}{\e} \rho \nabla_x  \left (\frac{\bar m}{\bar\rho} \right ) : \left (\frac{m}{\rho} - \frac{\bar m}{\bar\rho} \right ) \otimes \left (\frac{m}{\rho} - \frac{\bar m}{\bar\rho} \right )\, .
\end{align*}

Using \eqref{eq:elliptic} and integration by parts, we compute
\begin{align*}
 \frac{\X}{\e}\int_{\T^n}  &\dx  \left (\frac{\bar m}{\bar\rho} \right ) (c - \bar c) < \rho - \bar \rho > dx
 \\
 & =  
 \frac{\X}{\e}\int_{\T^n} \dx  \left (\frac{\bar m}{\bar\rho} \right ) (c - \bar c) ( \rho - \bar \rho )dx 
  -  \frac{\X}{\e} \int_{\T^n} \dx  \left (\frac{\bar m}{\bar\rho} \right ) \beta (c - \bar c)^2 dx \\
  &\quad - \frac{\X}{\e}\int_{\T^n} \dx  \left (\frac{\bar m}{\bar\rho} \right )  |\nabla_x(c - \bar c)|^2 dx
  - \frac{\X}{\e} \int_{\T^n} (c-\bar c) \nabla_x \left ( \dx  \left (\frac{\bar m}{\bar\rho} \right ) \right )  \cdot \nabla_x(c - \bar c) dx\, .
\end{align*}

From \eqref{defbarmep} and the smoothness of $\bar \rho$, we have that $\bar u = \frac{\bar m}{\bar \rho} =  O(\eps)$ and $\tfrac{1}{\eps} |\nabla_x \bar u | = O(1)$ as well as 
$\tfrac{1}{\eps} |\nabla^2_x \bar u | = O(1)$. Thus, we use Young's inequality to bound  the last term of the above relation as follows
\begin{equation*}
\frac{\X}{\e}\int_{\T^n} \left | (c-\bar c) \nabla_x \left ( \dx  \left (\frac{\bar m}{\bar\rho} \right ) \right )  \cdot \nabla_x(c - \bar c) \right | dx 
\leq C \int_{\T^n} \left ( (c-\bar c)^2 + |\nabla_x(c - \bar c)|^2 \right )dx \leq C \Phi(t) 
\end{equation*}
using \eqref{eq:ellipticreg}, \eqref{mainbound}, \eqref{eq:convexityresult} and the Poincar\'e inequality for the zero mean function $c-\bar c$  in $\T^n$ if $\beta = 0$, and where 
  $C>0$ (here and below) denotes  a generic positive constant. 
The estimates for the remaining terms in $Q$ are straightforward, again by virtue of \eqref{eq:ellipticreg}, \eqref{mainbound}, and in view of \eqref{eq:hyppress}
and \eqref{eq:convexityresult} we end up with  
$$ 
\int_0^t\int_{\T^n}|Q| dxd\tau \leq C\int_0^t\Phi(\tau)d\tau.
$$ 
Moreover, we recall $\bar e = O(\e)$ so that 
\begin{align*}
\int_0^t\int_{\T^n}|E| dxd\tau &\leq \frac{\e^2}{2}\int_0^t\int_{\T^n} \left | \frac{\bar e}{\bar\rho}\right |^2\rho dxd\tau + \frac{1}{2\e^2}\int_0^t\int_{\T^n}\rho \left |\frac{m}{\rho} - \frac{\bar m}{\bar \rho} \right |^2dxd\tau\\
&\leq CMT\e^4 + \frac{1}{2\e^2}\int_0^t\int_{\T^n}\rho \left |\frac{m}{\rho} - \frac{\bar m}{\bar \rho} \right |^2dxd\tau\, .
\end{align*}
Hence \eqref{last1} gives 
\begin{equation*}
  \Phi(t)  + \frac{1}{2\e^2}  \int_0^t\int_{\T^n}  \rho \left |\frac{m}{\rho} - \frac{\bar m}{\bar \rho}  \right |^2 dxd\tau  \leq \Phi(0) + C\e^4 + C\int_0^t\Phi(\tau)d\tau
\end{equation*}
and  Gronwall's Lemma gives the desired result.
\end{proof}

\section{From the Euler-Korteweg  system with friction to the  Cahn--Hilliard  equation}\label{sec:korteweg}
In this section we discuss the relaxation limit from the Euler--Korteweg system for the motion for capillary fluids  towards the Cahn--Hilliard equation. 
In our discussion we restrict to monotone pressure laws $p(\rho)$ and thus the model does not cover the case of phase transitions.
This model fits under the context of the general theory described in Section \ref{sec:largefrictvsgradflows} for a potential energy
\begin{equation*}
\mathcal{E}(\rho) = \int \left (h(\rho) +\tfrac12 C_\kappa |\nabla_x \rho |^2 \right ) dx, \quad C_\kappa>0\, ,
\end{equation*}
for which  the generator $\frac{\delta \cE}{\delta \rho }(\rho)$
takes the form 
\begin{equation*}
\frac{\delta \cE}{\delta \rho }(\rho) =\left.\big (\partial_\rho - \dx \partial_q\big )\left (h(\rho) +\tfrac12 C_\kappa |q |^2 \right )\right |_{q= \nabla_x \rho} =   
h'(\rho) - C_\kappa \triangle_x\rho\, ,
\end{equation*}
 the stress functional $S$ in \eqref{stresshyp} is given by
\begin{equation}
  S = \left (-p(\rho) +\frac12 C_\kappa |\nabla_x \rho|^2 + C_\kappa \rho
 \triangle_x \rho \right ) \I - C_\kappa \nabla_x \rho\otimes\nabla_x \rho\, ,
 \label{eq:defstressCH}
 \end{equation}
and $\I$ stands for the identity matrix.
Hence,  \eqref{eq:frictapprox} may be expressed as
\begin{equation}
    \begin{cases}
	\displaystyle{\rho_{t} +\frac{1}{\e}\dx m =0} &   \\[6pt]
       \displaystyle{ m_{t} + \frac{1}{\e}\dx \frac{m\otimes m}{\rho} 
	= -\frac{1}{\e^{2}}m - \frac{1}{\e}\rho \nabla_x \big (h'(\rho) - C_\kappa \triangle_x\rho \big)} & \\[6pt]
	\qquad \qquad \qquad \qquad \quad= \displaystyle{ -\frac{1}{\e^{2}}m + \frac{1}{\e}\dx S}\, ,& 
    \end{cases}
    \label{eq:Kortcap}
\end{equation}
where $t\in\R$, $x\in\T^{n}$, $\rho\geq 0$, $C_\kappa$ is a non negative constant, 
$m\in\R^{3}$, and the pressure $p(\rho)$ is defined by $p'(\rho) = \rho h''(\rho)$ and satisfies the monotonicity condition
 $p'(\rho) > 0$ and the stress tensor $S$ is given by \eqref{eq:defstressCH}.
As $\e\downarrow 0$, system \eqref{eq:Kortcap}, after an appropriate scaling of the momentum $m$,  formally reduces to the following Cahn--Hilliard equation
\begin{equation}
	\rho_{t} =\dx \left (\rho \nabla_{x}\big (h'(\rho) - C_\kappa \triangle_x\rho \big) \right )  = 
	\dx \big ( \nabla_{x}p(\rho) - C_\kappa\rho  \nabla_{x}\triangle_x\rho \big )\, .
	    \label{eq:cahnhill}
\end{equation}
Our goal is to rigorously validate this limit using the relative entropy identity \eqref{eq:relengfrelax} of Section \ref{sec:largefrictvsgradflows}.

We recall that the balance of total energy \eqref{eq:energy}, when expressed in the context of the system \eqref{eq:Kortcap}
takes the form:
\begin{align}
 &  \partial_t \left  ( \eta(\rho,m) +  \frac{C_\kappa}{2} |\nabla_x\rho|^2 \right )  +\frac{1}{\e}\dx  \big (q(\rho, m) - C_\kappa m \triangle_x \rho + C_\kappa \nabla_x \rho\dx m \big ) 
   \nonumber\\
   &\  =  
    - \frac{1}{\e^{2}}\frac{|m|^{2}}{\rho} \leq 0,\, 
    \label{eq:entrKort2}
\end{align}
where we used the notation
\begin{equation*}
   \eta(\rho,m) =  \frac{1}{2}\frac{|m|^{2}}{\rho} + h(\rho) \, , 
\qquad 
    q(\rho,m) = \frac{1}{2}m\frac{|m|^{2}}{\rho^{2}} + mh'(\rho)\, ,
\end{equation*}
for  the mechanical energy of the compressible Euler equations and its associated flux.
Moreover,  smooth solutions $\rho$ of \eqref{eq:cahnhill} satisfy the energy dissipation identity
\begin{align}
 &  \partial_t \left  ( h(\rho) +  \frac{C_\kappa}{2} |\nabla_x \rho|^2 \right )  
   -\dx  \Big [\rho h'( \rho) \nabla_{x}\big (h'(\rho) - C_\kappa \triangle_x \rho \big)  \nonumber\\
   &\ \ +C_\kappa \dx \left(\rho \nabla_{x}\big (h'(\rho) - C_\kappa \triangle_x\rho \big)  \right)    \nabla_x\rho 
   - C_\kappa\rho\nabla_{x}\big (h'(\rho) - C_\kappa \triangle_x \rho \big) \triangle_x \rho   \Big ] \nonumber\\
   &\ 
  =  - \rho \left |\nabla_x \left ( h'( \rho) - C_\kappa \triangle_x \rho \right) \right |^2
    \leq 0\, .
    \label{eq:entrcahnhill}
\end{align}
The above relations serve as  starting point to perform the relative energy computation in the next section.

\subsection{Relative energy estimate }\label{subsec:relenKorteasyrelax}
Next, we devise a relative energy estimate valid between a weak solution of \eqref{eq:Kortcap} and  a strong solution of \eqref{eq:cahnhill}.
This will be used to prove rigorously the diffusive relaxation limit for the case of constant capillarity coefficient $\kappa(\rho) = C_\kappa>0$ 
and for  $\gamma$--law gases $p(\rho) = k\rho^\gamma$.
To this end, following ideas from Section \ref{sec:KS} and \cite{LT12a, LT12b},  we rewrite the Cahn-Hilliard equation \eqref{eq:cahnhill} 
as a correction of the Euler-Korteweg system \eqref{eq:Kortcap},
\begin{equation}
    \begin{cases}
	\displaystyle{\bar\rho_{t} +\frac{1}{\e}\dx \bar m =0}&   \\
	& \\
       \displaystyle{ \bar m_{t} + \frac{1}{\e}\dx \frac{\bar m\otimes \bar m}{\bar\rho} 
     	= -\frac{1}{\e^{2}}\bar m} - \frac{1}{\e}\bar\rho \nabla_x \big (h'(\bar\rho) - C_\kappa \triangle_x\bar\rho \big)+ e(\bar \rho, \bar m) & \\[6pt]
	\qquad \qquad \qquad \qquad \quad= \displaystyle{ -\frac{1}{\e^{2}}\bar m + \frac{1}{\e}\dx \bar S + e(\bar \rho, \bar m)} \, ,& 
    \end{cases}
    \label{eq:Kortcapbar}
\end{equation}
 by introducing 
 $$
 \bar m = - \e \bar\rho \nabla_x \big (h'(\bar\rho) - C_\kappa \triangle_x\bar\rho \big) \, ,
 $$
and setting the error term $e(\bar \rho, \bar m)$ to be given by
\begin{align}
   \bar e :=  e(\bar \rho, \bar m) &= \frac{1}{\e} \dx \left (\frac{\bar 
       m\otimes \bar m}{ \bar\rho} \right ) +  \bar m_t  \nonumber\\
       &= 
       \e  \dx \Big (\bar\rho
       \nabla_{x} \big ( h'(\bar \rho) - C_\kappa \triangle_x\bar\rho\big )\otimes \nabla_{x} \big ( h'(\bar \rho) - C_\kappa \triangle_x\bar\rho \big )\Big ) 
       \nonumber\\ 
       &\ 
       -\e  \big ( \bar\rho \nabla_{x} \big ( h'(\bar \rho) - C_\kappa \triangle_x\bar\rho \big ) \big )_t  
       \nonumber\\
       &= O(\e).
       \label{eq:errortermKortcapbar}
\end{align}
Accordingly, \eqref{eq:entrcahnhill} is rewritten as
\begin{align}
   &\partial_t \left  ( \eta(\bar\rho,\bar m) +  \frac{C_\kappa}{2} |\nabla_x\bar\rho|^2 \right )  +\frac{1}{\e}\dx  \big (q(\bar\rho, \bar m) - C_\kappa \bar m \triangle_x \bar\rho + C_\kappa \nabla_x \bar\rho\dx\bar m \big )\nonumber\\
   &\ =  
    - \frac{1}{\e^{2}}\frac{|\bar m|^{2}}{\bar\rho} + \frac{\bar m}{\bar\rho}\cdot \bar e\, .
    \label{eq:entrKort2bar}
\end{align}
Following the general recipe \eqref{relpote} and \eqref{relkine}, the relative energy is defined via
\begin{equation*}
 \eta (\rho,m \left |  \bar \rho, \bar m\right. ) +
\frac{C_\kappa}{2}|\nabla_x(\rho - \bar\rho)|^2 = \frac{1}{2}\rho\left | \frac{m}{\rho} - \frac{\bar m}{\bar\rho}\right|^2 + h(\rho \left |  \bar \rho \right. ) + \frac{C_\kappa}{2}|\nabla_x(\rho - \bar\rho)|^2\, .
\end{equation*}

Before proving our main result, we recall the notion of weak solutions considered here, either conservative or dissipative \cite{GLT15}.
\begin{definition}  \label{def:weakCH}
(i)  A function $( \rho,  m)$ with  $\rho \in C([0, \infty) ; L^1 ( \T^n ) )$, $m \in C  \big ( ( [0, \infty) ;  \big ( L^1(\T^n) \big )^n \big )$,
$\rho \ge 0$,  is a weak
solution of \eqref{eq:Kortcap}   if 
$\frac{m \otimes m}{\rho}$, $S \in L^1_{loc}  \left ( ( (0, \infty) \times \T^n ) \right )^{n \times n}$,  and $(\rho, m)$ satisfy
\begin{align}
- \iint \Big[ \rho \psi_t + m \cdot \nabla_x \psi \Big]dx dt  = \int \rho (x, 0) \psi (x, 0) dx  \, , 
\\
- \iint \left [m \cdot \varphi_t + \frac{1}{\e}\left (\frac{m \otimes m}{\rho} : \nabla_x \varphi - S : \nabla_x \varphi \right)\right ] dx dt
\nonumber
\\
 = - \frac{1}{\e^2}\iint m\cdot \varphi dxdt +   \int m(x,0) \cdot \varphi(x,0) dx   \, ,
\end{align}
for all $\psi \in C^1_c \left ( [0, \infty) ; C^1_p (\T^n) \right )$ and $\varphi  \in C^1_c \left ( [0, \infty) ;  \big ( C^1_p (\T^n)  \big )^n \right )$.

\smallskip
(ii)  If,  in addition, $\eta(\rho,m) +\tfrac12 C_\kappa |\nabla_x \rho |^2 \in C([0, \infty) ; L^1 ( \T^n ) )$ and  $(\rho, m)$ satisfies
\begin{equation}
 \label{eq:Kortcap2Integr}
 \begin{aligned}
  - \iint \left ( \eta(\rho,m) +\tfrac12 C_\kappa |\nabla_x \rho |^2  \right ) \dot\theta(t) 
& \le   \int \left ( \eta(\rho,m) +\tfrac12 C_\kappa |\nabla_x \rho |^2 \right ) \Big |_{t=0}  \theta(0)dx 
 \\
 &\qquad  - \frac{1}{\eps^2} \iint \frac{|m|^2}{\rho} \theta(\tau) \, dx d\tau  
 \end{aligned}
\end{equation}
for any non-negative $\theta  \in W^{1, \infty} [0, \infty)$  compactly supported on $[0, \infty)$,  
then $(\rho, m)$ is called a dissipative weak solution. 

\smallskip
(iii)  By contrast,
if $\eta(\rho,m) +\tfrac12 C_\kappa |\nabla_x \rho |^2  \in C([0, \infty) ; L^1 ( \T^n ) )$ and it satisfies \eqref{eq:Kortcap2Integr}
as an equality, then $(\rho, m)$  is called a conservative weak solution.

\smallskip
(iv)  We say that a  dissipative (or conservative) weak periodic solution of \eqref{eq:Kortcap} $( \rho,  m)$    with $\rho \ge 0$
has finite total mass and energy if 
\begin{align*}
\sup_{t\in (0,T)} \int_{\T^n}  \rho  dx &\le K_1 < \infty\, ,
\\
 \sup_{t\in (0,T)} \int_{\T^n}
\left (\eta(\rho,m) +\tfrac12 C_\kappa |\nabla_x \rho |^2 \right )  \, dx
 &\le K_2 < \infty \, , 
\end{align*}
\end{definition}

The issue of existence of solutions for the Euler-Korteweg system \eqref{eq:Kortcap} is a subject of active current study and open in its full generality.
We refer to \cite{BDD07} for local well-posedness in Sobolev spaces, and to the recent paper  \cite{AH16} for global well-posedness of strong solutions
for small irrotational data  in dimensions greater than three. Results on global existence for finite energy solutions are available for the system 
of quantum hydrodynamics (\cite{AM09,AM12}, \cite{GM97})
due to its special connection to the linear Schroedinger equation via the Madelung transformation. 
Finally, we refer to \cite{DFM15}, \cite{GLT15} for various well or ill posedness results applying to
general Euler-Korteweg systems.


Following  \cite[Section~3.2]{GLT15}, we  prove a relative energy estimate between a weak periodic solution of \eqref{eq:Kortcap} 
and a strong periodic solution of \eqref{eq:cahnhill}.  We point out that that the main difference
between the present case and  the analysis in \cite[Section~3.2]{GLT15} is 
due to the presence of the friction term and the error term induced by embedding the Cahn-Hilliard equation within the Euler-Korteweg system,
and is manifested in the relative kinetic energy identity analyzed below.
It is worth noting that the regularity requirements needed to justify the calculation below and the definition of the stress $S$ in \eqref{eq:defstressCH} are
satisfied, since the uniform bound for the energy yields
 \begin{equation*}
 \rho\in C([0,T];L^\gamma(\T^n))\ \hbox{and}\ \nabla_x \rho \in C([0,T]; L^2(\T^n))\, 
 \end{equation*}
Note that $\big(\rho F_q\big)(t, \cdot) = \big(\rho \nabla_x\rho\big)(t, \cdot)  \in L^{1}(\T^n) $ since $\rho(t, \cdot)\in L^{2}(\T^n)$,  and the latter property
 comes from $\rho(t, \cdot)\in L^{1}(\T^n)\cap L^{\gamma}(\T^n)$ and, if $\gamma<2$, 
the $L^2$ integrability of $\nabla_x \rho$ and the Gagliardo--Niremberg--Sobolev inequality.  For further details, see \cite{GLT15} and in particular  condition \textbf{(A)} of page 22 and Remark~3.3.
\begin{theorem}\label{theo:relenweakCH}
Let   $(\rho,m)$ be a dissipative or conservative weak solution of  \eqref{eq:Kortcap}  with finite total energy according to Definition \ref{def:weakCH}, and let $\bar\rho$ be a smooth solution  of \eqref{eq:cahnhill}. Then
\begin{align}
\label{eq:RelEnKorGenFinalweakCH}
& \int_{\T^n} \left. \left ( \eta(\rho,m | \, \bar \rho, \bar m \right. ) + \tfrac12 C_\kappa | \nabla_x( \rho - \bar\rho)|^2 \right )dx \Big |_t \nonumber\\
&\ \leq
\int_{\T^n} \left. \left ( \eta(\rho,m | \, \bar \rho, \bar m \right. ) + \tfrac12 C_\kappa | \nabla_x( \rho - \bar\rho)|^2 \right )dx \Big |_0
\nonumber\\
  &\ \  -\frac{1}{\e^2}  \iint_{[0,t]\times\T^n}  \rho \left |\frac{m}{\rho} - \frac{\bar m}{\bar \rho}  \right |^2 dxd\tau  - \iint_{[0,t]\times\T^n} e(\bar \rho, \bar m)\cdot \frac{\rho}{\bar\rho}\left ( \frac{m}{\rho} - \frac{\bar m}{\bar \rho} \right )  dxd\tau
\nonumber\\
&\ \ - \frac{1}{\e}\iint_{[0,t]\times\T^n}  \dx \left (\frac{\bar m}{\bar\rho}\right )\Big [   p(\rho \left  | \bar \rho \right .) + \tfrac12 C_\kappa | \nabla_x( \rho - \bar\rho)|^2 \Big ]dxd\tau
\nonumber\\
&\ 
- \frac{C_\kappa}{\e}\iint_{[0,t]\times\T^n} \left [ \nabla_x \left (\frac{\bar m}{\bar\rho} \right) : \nabla_x(\rho - \bar\rho)\otimes\nabla_x(\rho - \bar\rho)\right.
\nonumber\\
&\ \ \left.
+ \nabla_x \dx \left (\frac{\bar m}{\bar\rho}\right ) \cdot (\rho-\bar\rho)\nabla_x(\rho-\bar\rho)\right]dxd\tau
\nonumber\\
&\ \   - \frac{1}{\e}\iint_{[0,t]\times\T^n}  \rho \nabla_x  \left(    \frac{\bar m}{\bar\rho}\right )  :  \left( \frac{m}{\rho}-   \frac{\bar m}{\bar\rho}\right )\otimes \left( \frac{m}{\rho}-   \frac{\bar m}{\bar\rho}\right ) dxd\tau\, .
\end{align}
\end{theorem}

\begin{proof} Let $(\rho, m)$  a weak dissipative (or conservative) solution of \eqref{eq:Kortcap}, let $(\brho , \bar m)$  be a strong 
 solution or \eqref{eq:Kortcapbar}, and as already noted embed the Cahn-Hilliard equation into the Euler-Korteweg system
 with an error term \eqref{eq:Kortcapbar}.
 In \eqref{eq:Kortcap2Integr} we consider $\theta(\tau)$ given in  \eqref{testthetaS}
and let  $\mu \downarrow 0$ to conclude
\begin{equation}
\label{lem:ret1CH}
  \left. \int_{\mathbb{T}^n}\Big ( \eta(\rho,m) + \tfrac12 C_\kappa|\nabla_x\rho|^2) \Big)dx  \right |^{t}_{\tau=0}
\leq    -  \frac{1}{\eps^2} \iint_{[0,t]\times\T^n} \frac{|m|^2}{\rho} \, dx d\tau  
\end{equation}
Moreover, time integration of \eqref{eq:entrKort2bar} gives
\begin{equation}\label{lem:retCH}
\begin{aligned}
&\left.  \int_{\T^n}   \Big (\eta(\bar\rho,\bar m)  +    \tfrac{1}{2} C_\kappa  \, |\nabla_x\bar\rho|^2\Big )dx \right |^{t}_{\tau=0}  
\\
&\quad= -  \frac{1}{\e^{2}}\iint_{[0,t]\times\T^n}\frac{|\bar m|^{2}}{\bar\rho} dx d\tau  + 
   \iint_{[0,t]\times\T^n} \frac{\bar m}{\bar\rho}\cdot \bar e  dx d\tau \, .
\end{aligned}
\end{equation}

We evaluate now the dynamics of the linear part of the relative energy, starting from the  weak formulation 
for the equations satisfied by the differences $(\rho - \bar\rho, m- \bar m)$:
\begin{align}
- \iint_{[0,+\infty)\times\T^n} \Big ( {\psi}_t (\rho - \bar \rho) + \frac{1}{\e} \psi_{x_i} (m_i - \bar m_i)  \Big ) dx dt
- \int_{\mathbb{T}^n} \psi (\rho - \bar \rho) \Big |_{t=0} \, dx = 0\, ,
\label{weakmassCH}
\\
- \iint_{[0,+\infty)\times\T^n} \left (\varphi_t \cdot (m - \bar m) + \frac{1}{\e} \partial_{x_i}\varphi_j  \left (\frac{m_i m_j}{\rho} - \frac{\bar m_i \bar m_j}{\bar\rho} \right)
- \frac{1}{\e}\partial_{x_i} \varphi_j \Big(S_{ij} - \bar S_{ij} \Big)  \right ) dx dt
\nonumber
\\
- \int_{\mathbb{T}^n} \varphi \cdot (m - \bar m) \Big |_{t=0}  dx 
= - \frac{1}{\e^2}\iint_{[0,+\infty)\times\T^n} (m - \bar m)\cdot \varphi dxdt  - \iint_{[0,+\infty)\times\T^n} \bar e\cdot \varphi dxdt \, ,
\label{weakmomentumCH}
\end{align}
where, as before, $\varphi$, $\psi$ are Lipschitz test functions 
compactly supported in $ [0,+\infty)$ in time  and periodic in space.

In the above relations, we  select  the test functions:
$$
\psi = \theta(\tau)\left( h'(\bar\rho)  - C_\kappa \triangle_x\bar\rho -\frac{1}{2}\frac{|\bar m|^2}{\bar\rho^2}   \right ) , \quad
\varphi = \theta(\tau) \frac{\bar m}{\bar\rho} \,  ,
$$
with $\theta(\tau)$ as in \eqref{testthetaS}.
Then \eqref{weakmassCH} with  $\mu\downarrow    0$ gives the linear part of the potential energy:
\begin{align*}
& \int_{\mathbb{T}^n}  \Big (h'(\bar\rho)  (\rho - \brho) + C_\kappa \nabla_x\bar\rho\cdot\nabla_x (\rho - \bar \rho)   -\frac{1}{2}\frac{|\bar m|^2}{\bar\rho^2}  (\rho - \bar \rho) 
  \Big ) \Bigg |_{\tau=0}^t dx 
\nonumber\\
&\quad - \iint_{[0,t]\times\T^n} \left [\partial_\tau\left( h'(\bar\rho)   -\frac{1}{2}\frac{|\bar m|^2}{\bar\rho^2}   \right ) (\rho - \bar \rho) 
+ C_\kappa\partial_\tau\big (\nabla_x\bar\rho\big) \cdot  \nabla_x (\rho - \bar \rho) \right]dx d\tau 
\nonumber\\
&\quad -\frac{1}{\e} \iint_{[0,t]\times\T^n} \nabla_x\left(  h'(\bar\rho)  - C_\kappa \triangle_x\bar\rho-\frac{1}{2}\frac{|\bar m|^2}{\bar\rho^2}   \right ) \cdot (m- \bar m)   dx d\tau = 0\, .
\end{align*}

Moreover, from \eqref{weakmomentumCH} one has 

\begin{align*}
& \int_{\mathbb{T}^n} \frac{\bar m}{\bar\rho}\cdot (m - \bar m)  \Big |_{\tau=0}^tdx
- \iint_{[0,t]\times\T^n} \partial_\tau\left( \frac{\bar m}{\bar\rho}   \right ) \cdot (m-\bar m) dx d\tau 
\nonumber\\
&\ \ - \frac{1}{\e}
\iint_{[0,t]\times\T^n} \partial_{x_i}\left( \frac{\bar m_j}{\bar\rho}   \right ) \left  ( \Big ( \frac{m_im_j}{\rho}- \frac{\bar m_i\bar m_j}{\bar\rho}  \Big )
  -  ( S_{ij} - \bar S_{ij} )  \right )  dx d\tau 
  \nonumber\\ 
  &\ = -\frac{1}{\e^2}\iint_{[0,t]\times\T^n} \frac{\bar m}{\bar\rho} \cdot ( m - \bar m )dxd\tau    -  \iint_{[0,t]\times\T^n} \frac{\bar m}{\bar\rho}\cdot \bar e dx d\tau\, .
\end{align*}

Combining  the above equations we conclude
\begin{align}
& \int_{\mathbb{T}^n} \Big [ \eta(\rho, m \left | \bar \rho, \bar m \right.  ) 
+ \tfrac{1}{2}\C|\nabla_x(\rho - \bar \rho)|^2  \Big ]\Bigg |_{\tau=0}^t dx 
\nonumber\\
&\ \leq - \frac{1}{\e^2}\iint_{[0,t]\times\T^n} \Big [  \rho|u|^2  - \bar\rho|\bar u|^2   - \bar u \cdot  ( \rho u - \bar \rho\bar u ) \Big ]dx d\tau
\nonumber\\
&\ \ - \iint_{[0,t]\times\T^n} \Big [ \partial_\tau\big(  h'(\brho)  - \tfrac{1}{2} |\bar u|^2  \big ) (\rho - \bar \rho)  
+ C_\kappa\partial_\tau\big (\nabla_x\bar\rho\big) \cdot  \nabla_x (\rho - \bar \rho)
\nonumber\\
&\qquad
+ \partial_\tau(\bar u )\cdot (\rho u - \bar\rho \bar u)  \Big ] dxd\tau 
\nonumber\\
&\ \ - \frac{1}{\e} \iint_{[0,t]\times\T^n} \nabla_x\big(  h'(\brho)  - C_\kappa \triangle_x\bar\rho - \tfrac{1}{2}|\bar u|^2 \big ) \cdot (\rho u- \bar \rho \bar u)   dx d\tau
\nonumber\\
&\ \ - \frac{1}{\e} \iint_{[0,t]\times\T^n} \partial_{x_i} ( \bar u_j )  \Big (  ( \rho u_i u_j -  \bar \rho \bar u_i \bar u_j   )
  - (S_{ij} - \bar S_{ij})     \Big )   dx d\tau
  \, ,
\label{eq:int1CH}
\end{align}
where $m = \rho u$ and $\bar m = - \e \bar\rho \nabla_x \big (h'(\bar\rho) - C_\kappa \triangle_x\bar\rho \big) = \bar\rho \bar u $.
For this model, $\bar u$ verifies
\begin{equation*}
\partial_t \bar u + \frac{1}{\e}( \bar u \cdot \nabla_x ) \bar u  = -  \frac{1}{\e^2} \bar u  - \frac{1}{\e} \nabla_x \big (h'(\bar\rho) - C_\kappa \triangle_x\bar\rho \big) + \frac{\bar e}{\bar \rho}\, ,
\end{equation*}
so that
\begin{align*}
&\partial_\tau \big (- \tfrac{1}{2} |\bar u|^2  \big )(\rho - \bar \rho) + \partial_\tau(\bar u )\cdot (\rho u - \bar\rho \bar u) + \frac{1}{\e} 
\nabla_x\big( - \tfrac{1}{2}|\bar u|^2 \big) \cdot (\rho u- \bar \rho \bar u)
\\
&\ +\frac{1}{\e}  \partial_{x_i} ( \bar u_j )     ( \rho u_i u_j -  \bar \rho \bar u_i \bar u_j   ) = -\frac{1}{\e^2}\rho \bar u\cdot(u-\bar u) - \frac{1}{\e}\rho  \nabla_x \big (h'(\bar\rho) - C_\kappa \triangle_x\bar\rho \big)
\cdot(u-\bar u) 
\\
&\ + \bar e \frac{\rho}{\bar\rho}\cdot(u-\bar u)  -\frac{1}{\e} \rho \nabla_x    (\bar u) :   (u - \bar u) \otimes  (u - \bar u)\, .
\end{align*}
Then, as in the previous section,  
using also the  reformulation of the stress given by \eqref{eq:defstressCH},
$$
\dx S = - \rho \nabla_x \big (h'(\rho) - C_\kappa \triangle_x\rho \big) \, ,
$$
we rewrite \eqref{eq:int1CH} as follows:
\begin{align}
& \int_{\mathbb{T}^n} \Big [ \eta(\rho, m \left | \bar \rho, \bar m \right.  ) 
+ \tfrac{1}{2}\C|\nabla_x(\rho - \bar \rho)|^2    \Big ]\Bigg |_{\tau=0}^t dx 
\nonumber\\
&\ \leq - \frac{1}{\e^2}\iint_{[0,t]\times\T^n} \rho |u - \bar u|^2dx d\tau - \iint_{[0,t]\times\T^n} \bar e \frac{\rho}{\bar\rho}\cdot(u-\bar u) dx d\tau
\nonumber\\
&\ \ - \frac{1}{\e}\iint_{[0,t]\times\T^n}  \rho \nabla_x    (\bar u) :   (u - \bar u) \otimes  (u - \bar u)dx d\tau 
- \frac{1}{\e}\iint_{[0,t]\times\T^n} \dx (\bar u) p(\rho \left | \bar \rho \right.  )dx d\tau
\nonumber\\
&\ \ -
 \frac{1}{\e} \C\iint_{[0,t]\times\T^n} \left [ \tfrac12 \dx (\bar u) \Big( |\nabla_x \rho|^2 - |\nabla_x \bar\rho|^2 \Big) + \nabla_x\dx( \bar u)\cdot \Big(\rho\nabla_x\rho -  
\bar \rho\nabla_x\bar\rho\Big)\right ]dx d\tau
\nonumber\\
&\ \ -
 \frac{1}{\e} \C\iint_{[0,t]\times\T^n}\nabla_x    (\bar u) : \Big[ \nabla_x\rho\otimes \nabla_x\rho - \nabla_x\bar\rho\otimes \nabla_x\bar\rho   \Big] dxd\tau
\nonumber\\
&\ \ +
 \frac{1}{\e} \C\iint_{[0,t]\times\T^n}\Big [ \nabla_x\dx (\bar \rho \bar u) \cdot \nabla_x(\rho - \bar\rho) + \bar u \cdot (\rho - \bar\rho) \nabla_x\triangle_x\bar\rho\Big ] dxd\tau \, .
\label{eq:int2CH}
\end{align}
Then we compute
\begin{align*}
 \iint_{[0,t]\times\T^n}  \nabla_x\dx (\bar \rho \bar u) \cdot \nabla_x(\rho - \bar\rho)  dxd\tau 
&= 
\iint_{[0,t]\times\T^n} 
 \Big[  \nabla_x\dx (    \bar u) \cdot \bar \rho   \nabla_x(\rho - \bar\rho)  
\\
&\  +    \dx (    \bar u)\nabla_x \bar \rho \cdot    \nabla_x(\rho - \bar\rho)  
 + \nabla_x \big( \bar u \cdot \nabla_x\bar\rho\big)\cdot \nabla_x(\rho - \bar\rho) 
 \Big ]dxd\tau
\end{align*}
and integration by parts  gives
\begin{align*}
 \iint_{[0,t]\times\T^n}\bar u \cdot (\rho - \bar\rho) \nabla_x\triangle_x\bar\rho   dxd\tau &=
  \iint_{[0,t]\times\T^n} \nabla_x\dx \big((\rho - \bar\rho)\bar u \big) \cdot \nabla_x\bar\rho dxd\tau
 \\
 &= 
 \iint_{[0,t]\times\T^n}\Big[  \nabla_x\dx (    \bar u) \cdot (\rho -\bar \rho )  \nabla_x \bar\rho
 +    \dx (    \bar u)\nabla_x \bar \rho \cdot    \nabla_x(\rho - \bar\rho)  
 \\ 
 &\  + \nabla_x \big( \bar u \cdot \nabla_x(\rho -\bar\rho)\big)\cdot \nabla_x\bar\rho
 \Big ]dxd\tau\, .
\end{align*}
Moreover, 
\begin{align*}
&\iint_{[0,t]\times\T^n} 
     \dx (    \bar u)\nabla_x \bar \rho \cdot    \nabla_x(\rho - \bar\rho)  
 + \nabla_x \big( \bar u \cdot \nabla_x\bar\rho\big)\cdot \nabla_x(\rho - \bar\rho) + \nabla_x \big( \bar u \cdot \nabla_x(\rho -\bar\rho)\big)\cdot \nabla_x\bar\rho
 \Big ]dxd\tau 
 \\
 &\ =
  \iint_{[0,t]\times\T^n} \partial_{x_i}\bar u_j \Big [ \partial_{x_j}\bar\rho\partial_{x_i}(\rho - \bar\rho) + 
  \partial_{x_j}(\rho -\bar\rho)\partial_{x_i}\bar\rho
  \Big ]dxd\tau 
   \\
 &\ \
  + 
  \iint_{[0,t]\times\T^n} \bar u_j \Big [\partial_{x_i} \partial_{x_j}\bar\rho\partial_{x_i}(\rho - \bar\rho) + 
 \partial_{x_i} \partial_{x_j}(\rho -\bar\rho)\partial_{x_i}\bar\rho
    \Big ]dxd\tau 
  + 
  \iint_{[0,t]\times\T^n} \partial_{x_j}\bar u_j \Big [ \partial_{x_i}\bar\rho\partial_{x_i}(\rho-\bar\rho)
   \Big ]dxd\tau 
   \\
 &\ = 
 \iint_{[0,t]\times\T^n} \partial_{x_i}\bar u_j \Big [  \partial_{x_j}\bar\rho\partial_{x_i}\rho + \partial_{x_j}\rho\partial_{x_i}\bar\rho - 2 \partial_{x_j}\bar\rho\partial_{x_i}\bar\rho
  \Big ]dxd\tau \, .
\end{align*}
Finally, using the above relations in \eqref{eq:int2CH} we end up with \eqref{eq:RelEnKorGenFinalweakCH} and  the proof is complete.
\end{proof}

\subsection{Convergence to Cahn-Hilliard in the large friction limit}\label{convCH}
Having established \eqref{eq:RelEnKorGenFinalweakCH}, we  proceed as in the previous section to obtain 
a stability estimate and complete the convergence in the diffusive relaxation limit.
Let us introduce the following quantity, based on the relative energy,
\begin{equation}
\label{defpsi}
\Psi(t) = \int_{\T^n} \Big [\eta(\rho, m \left | \bar \rho, \bar m \right.  ) 
+ \tfrac{1}{2}\C|\nabla_x(\rho - \bar \rho)|^2 \Big ]dx 
\end{equation}
to control   the distance between the  solutions of \eqref{eq:Kortcap} and \eqref{eq:Kortcapbar}. As it is manifest from Lemma~\ref{lem:generalconvEuler3d} and 
 Remark~\ref{rem:hnorm}, if $\gamma\geq2$ the relative potential energy, and thus $\Psi$, in particular controls the $L^2$ norm of the distance $\rho - \bar\rho$; see \eqref{eq:hnorm2}. 
Hence, all the terms on the right and side of \eqref{eq:RelEnKorGenFinalweakCH} are estimated in terms of $\Psi(t)$
and we can follow the lines of the proof of Theorem~\ref{th:finalKS} to prove:

\begin{theorem}\label{th:finalCH}
Let  $T>0$ be fixed, let $p(\rho)$ satisfy \eqref{hypthermo}, \eqref{hyppress} and \eqref{eq:growthh} with $\gamma\geq2$,  let  $(\rho,m)$ be a dissipative (or conservative) weak solution of  \eqref{eq:Kortcap}  with finite total energy according to Definition \ref{def:weakCH}, and let $\bar\rho$ be a smooth solution  of \eqref{eq:cahnhill} such that $\bar\rho\geq \delta >0$. Then  the stability estimate
\begin{equation*}
\Psi(t) \leq C \big (\Psi(0) + \e^4 \big ), \quad t\in [0,T] \, ,
\end{equation*}
holds true,  with $C$  a positive constant depending only on $T$, $K_1$, $\bar\rho$ and its derivatives.
Moreover, if $\Psi(0)\to 0$ as $\e \downarrow 0$, then
\begin{equation*}
\sup_{t\in[0,T]}\Psi(t) \to 0, \ \hbox{as}\ \e \downarrow 0\, .
\end{equation*}
\end{theorem}
\begin{remark}\label{rem:finalCH}
The convergence in the general case $\gamma>1$ can also be established, provided a uniform $L^\infty$ bound for both $\rho$ and $\bar\rho$ is available. 
In the particular case $\gamma=2$, we have $h(\rho \left | \bar \rho \right.)=|\rho-\bar\rho|^2$ and  Lemma~\ref{lem:generalconvEuler3d}  is not needed anymore;
in that case we may treat the presence of vacuum  for both  $\rho$ and $\bar\rho$. However, the regularity assumptions for the latter are still necessary, 
and this might be inconsistent  with the presence of vacuum for $\bar\rho$; see also \cite{GLT15}.
\end{remark}

\medskip
\noindent
{\bf Acknowledgements.} 
We thank the anonymous referee for suggesting
to consider the effects of confinement potentials, what led to Section \ref{sec:confinement}.
AET was supported by funding from King Abdullah University of Science and Technology (KAUST).

\end{document}